\documentclass[12pt]{article}
\usepackage{amssymb,amsmath,amsthm,sectsty,url}
\usepackage[letterpaper,hmargin=1.0in,vmargin=1.0in]{geometry}
\usepackage{color}
\usepackage[boxed]{algorithm}
\usepackage[margin=20pt,font=small,labelfont=bf]{caption}
\floatname{algorithm}{Figure}

\sectionfont{\large}
\subsectionfont{\normalsize}
\numberwithin{equation}{section}

\newtheorem{maintheorem}{Theorem}
\newtheorem{theorem}{Theorem}[section]
\newtheorem{lemma}[theorem]{Lemma}
\newtheorem{proposition}[theorem]{Proposition}
\newtheorem{corollary}[theorem]{Corollary}

\newtheorem{definition}[theorem]{Definition}
\newtheorem{remark}[theorem]{Remark}

\newtheorem{claim}[theorem]{Claim}
\newtheorem{question}[theorem]{Question}

\renewcommand{\Pr}{ \mathrm P}

\newcommand{\ga}{\alpha}

\newcommand{\G}{{\cal G}}

\newcommand{\N}{\mathbb N}
\newcommand{\R}{\mathbb R}
\newcommand{\Z}{\mathbb Z}

\newcommand{\ehl}{\color{black}}

\begin{document}

\title{Infinite and Giant Components in the Layers Percolation Model}
\author{Jonathan Hermon
\thanks{Department of Statistics, UC Berkeley. E-mail: {\tt jonathan.hermon@berkeley.edu}.}}
\date{}
\maketitle

\begin{abstract}
 In this work we continue the investigation launched in \cite{feige2013layers} of the structural properties of the structural properties of the \emph{Layers model}, a dependent percolation model. Given an undirected graph $G=(V,E)$ and an integer $k$, let $T_k(G)$
denote the random vertex-induced subgraph of $G$, generated by ordering $V$
according to Uniform$[0,1]$ $\mathrm{i.i.d.}$ clocks and including in $T_k(G)$ those vertices
with at most $k-1$ of their neighbors having a faster clock. The distribution
of subgraphs sampled in this manner is called the \emph{layers model with
parameter} $k$. The layers model has found applications in the study of $\ell$-degenerate
subgraphs, the design of algorithms for the maximum independent set problem
and in the study of bootstrap percolation.

We prove that every infinite locally finite  tree $T$ with no leaves, satisfying that the degree of the vertices grow sub-exponentially in their distance from the root, $T_3(T)$ $\mathrm{a.s.}$ has an infinite connected component. In contrast, we show that for any locally finite graph $G$, $\mathrm{a.s.}$ every connected component of $T_2(G)$ is finite.

We also consider random graphs with a given degree sequence and show that if the minimal degree is at least 3 and the maximal degree is bounded, then $\mathrm{w.h.p.}$ $T_3$ has a giant component. Finally, we also consider ${\Z}^{d}$ and show that if $d$ is sufficiently large, then $\mathrm{a.s.}$ $T_4(\Z^d)$ contains an infinite cluster.
\ehl
\end{abstract}

\newpage


\section{Introduction}
Consider the following percolation model. Given a  graph $G=(V,E)$, every vertex $v$ of
$G$ selects independently at random an ``age" $X_v$ from the uniform distribution
$\mathrm{Uniform}[0,1]$. For any $k \in \N$, define $L_k(G)$ to be the set of
those vertices that have exactly $k-1$ younger neighbors. For an integer $k \ge 1$, we call $L_k(G)$ the $k$th \emph{layer} of $G$. The union of the first $k$ layers
is denoted by $T_k(G): = \bigcup_{i=1}^k L_i(G)$. By a slight abuse
of notation we refer to the subgraph induced on $T_k$ also by $T_k$, and
omit $G$ when clear from context.

Note that if a vertex $v$ has $m$ neighbors in $G$, then for any $1 \le i \le m+1$ we have that $v \in L_i(G)$ with probability $\frac{1}{m+1}$. However,  these events are  not independent for different vertices  of distance at
most 2. As standard in percolation models, we say that $v \in V$ is \emph{open} (\emph{closed}) if $v \in T_{k}(G)$ (respectively, $v \notin T_k(G)$), where $k$ is clear from context.

The above procedure for sampling vertices from a graph has several useful
properties which were exploited in the design of algorithms
for finding large independent sets in graphs (e.g.~$\cite{feige2013recoverable}$) and in the study of \emph{contagious sets} for \emph{bootstrap percolation} (e.g.~$\cite{reichman2012new}$). For a list of algorithmic applications of the layers model see the related work section in
$\cite{feige2013layers}$.
In 
$\cite{feige2013layers}$ the treewidth and the size of the largest connected component of $T_3(G)$ and $T_2(G)$ were analyzed for various graphs. 

In this paper we establish
parallel versions of the aforementioned results  from  $\cite{feige2013layers}$ for infinite graphs. We also generalize a theorem in $\cite{feige2013layers}$ concerning random 3-regular graphs to random graphs with more general degree sequences (see Theorem \ref{thm : randomgraphs}). 

We denote by $T=(V,E,o)$ a rooted-tree with root $o$. This is simply a tree
with some distinct vertex, denoted by $o$. The $r$th level of $T$, denoted by $\ell_r(T)$, is the collection of vertices of $T$ which are at distance $r$ from $o$. We denote the degree of $v \in V$ by $d_v$. In Section \ref{sec: trees} we study $T_3$ on infinite trees and prove the following Theorem.
\begin{maintheorem}
\label{thm: mainresulttreesinto}
Let $T=(V,E,o)$ be a rooted-tree of minimal degree at least 3. If there exist $C >0$ and $1 \le a < \sqrt{4/3}$, such that $\max_{v \in\ell_r(T)}d_v \le Ca^r$ for all $r \ge 0$, then $T_3(T)$ has an infinite cluster $\mathrm{a.s.}$.
\end{maintheorem}
\begin{remark}
As the following example demonstrates, the degree growth condition in Theorem 1 is necessary, up to the value of $a$.
 Denote $a_n:=2^{2^{2^{n}}}$.
Consider a spherically symmetric tree $T$ rooted at $o$, in which every $v \in \ell_r(T)$ has degree $3-1_{r=0}$ if $r \notin \{a_{n}:n \in \N \} $ and otherwise every $v \in \ell_r(T)$ has degree $|\ell_r(T)|+1$. Then  for every $n \in \N $ we have that $|\ell_{a_n}(T)| = 2^{\sum_{i=1}^na_{i}} =2^{a_{n}(1+o(1))} $. Let $A_n$ be  the event that $\ell_{a_n}(T) \cap T_3 = \emptyset $. By the Borel-Cantelli Lemma $\mathrm{a.s.}$ infinitely many of the events $A_n$ occur, which implies that there are no infinite clusters in $T_3$. 
\end{remark}
It was shown in 
$\cite{feige2013layers}$ that for a bounded degree finite graph $G$ of size
$n$, the largest connected component of $T_2(G)$ is $\mathrm{w.h.p.}$\footnote{With high probability - that is, with probability tending to 1 as $n \to \infty$.} of size at most
$O(\log n)$. In Section \ref{sec: T2} we prove an analogous result in the infinite setup, while dropping the assumption that $G$ is of bounded degree.
\begin{maintheorem}
\label{thm: T2isfinitenonbdddegreecaseintro}
Let $G=(V,E)$ be a locally finite graph with a countable vertex set $V$.
 Then, $$\Pr[T_2(G) \text{ has an infinite connected component}]=0.$$
\end{maintheorem}

The only infinite graph considered in $\cite{feige2013layers}$
 is $\Z^2$, for which it was shown that $T_4(\Z^2)$  $\mathrm{a.s.}$ has a unique
infinite connected component (which we also call an \emph{infinite cluster}).

In Section \ref{sec: Zd} we consider $T_{4}(\Z^d)$  for $d>2$ and prove the following theorem.
\begin{maintheorem}
\label{thm: Z^dintro}
 $T_4(\Z^{d})$ has an infinite cluster $\mathrm{a.s.}$, for all sufficiently large  $d$.
\end{maintheorem}
In fact, we prove a stronger assertion than that of Theorem \ref{thm: Z^dintro}. Theorem \ref{thm: Z^d} asserts that  $T_4(\Z^{d})$ contains an open infinite monotone path\footnote{A path w.r.t.~an oriented copy of $\Z^d$.}  $\mathrm{a.s.}$, for all sufficiently large 
$d$. We expect Theorem \ref{thm: Z^dintro} to hold for all $d$, however it seems that even with a more careful analysis, the argument in the proof of Theorem \ref{thm: Z^dintro} cannot be used to prove this, because of the restriction of the argument to monotone paths.
\begin{question}
Does $T_3(\Z^d)$ $\mathrm{a.s.}$ contain a unique infinite connected component for all $d \ge 1$? 
\end{question}
In \cite{feige2013layers} it was shown that for random 3-regular graphs, $\mathrm{w.h.p.}$\footnote{Over the joint distribution of the random 3-regular graph $H$ and $T_3(H)$.} $T_3$ contains a giant component\footnote{
A giant component is a connected component of size $\Omega(n)$}. In Section \ref{sec: expanders} we generalize the aforementioned
result by considering random graphs with more general degree
sequences. Denote $[n]:=\{1,\ldots,n \}$.
\begin{maintheorem}
\label{thm : randomgraphs}
Let $d \ge 3$. Let $G=([n],E)$ be a random graph chosen from the uniform distribution over all  labeled graphs satisfying that the degree of vertex $i$ equals $d_i$ for some sequence of numbers $(d_i : 1 \le i \le n)$ such that  $\sum_i d_i$ is even and  $3 \le d_i \le d$, for all $i \in [n]$. Then there exists a constant $\alpha=\alpha(d)$ such that
$$\Pr[T_3(G) \text{ contains a connected component of size at least } \alpha n]=1-\exp(-\Omega( n)),$$
where the probability is taken jointly over the choice of $G$ and of the ages of the vertices. 
\end{maintheorem} 
Note that in Theorems \ref{thm: mainresulttreesinto} and \ref{thm: T2isfinitenonbdddegreecaseintro}
we do not assume that the graphs are of bounded degree. Similarly, in Theorem \ref{thm : randomgraphs} we allow fairly general degree sequences. This aspect of Theorems \ref{thm:
mainresulttreesinto}, \ref{thm: T2isfinitenonbdddegreecaseintro} and \ref{thm
: randomgraphs} is interesting, since a-priori, it is not obvious how to construct a canonical site percolation process, that has the same marginal probabilities as the Layers model, which exhibits such behaviors.

It is interesting to note that ``3" is the critical value for the Layers model, both in the setup of Theorem \ref{thm: mainresulttreesinto}
and of Theorem \ref{thm : randomgraphs} (in the sense that $T_3$ has an infinite cluster ($\mathrm{a.s.}$ in the infinite setting) or a giant component ($\mathrm{w.h.p.}$ in the finite setting)). This is not a coincidence.
Our proof of Theorem \ref{thm : randomgraphs}  rely heavily
on the analysis of the layers model on trees. 
It is well-known that, in some sense, the random graphs considered in Theorem \ref{thm : randomgraphs}, locally look like trees.  
Thus it is natural to interpret the coincidence 
of the critical value ``3"  in  Theorems \ref{thm: mainresulttreesinto} and \ref{thm : randomgraphs} as an instance in which the critical parameter is a ``local property". Such locality is conjectured to hold in greater generality for Bernoulli (independent) Percolation \cite{benjamini2011critical} (the results in \cite{benjamini2011critical} cover  in particular the case of large girth expanders).

\subsection{An overview of our techniques}
\label{sec: overview}
We now present a short informal discussion which summarizes the main ideas in this paper. 

Let $G=(V,E)$ be a graph and $k \in \N$. Let $Y_v$ be the indicator of the event that vertex $v$ belongs to $T_k(G)$. Clearly, if $A \subset V$ and for any distinct $v,u \in A$ we have that $u$ and $v$ are not neighbors and do not have a common neighbor, then $(Y_v)_{v \in A}$ are independent. However, if there exist $u,v \in A$ which are neighbors or have a common neighbor, then typically $(Y_v)_{v \in A}$ are not independent.

It is hard to analyze the possible global effects of such dependencies. To avoid this, it is useful to consider a family of paths in $G$, $\Gamma$, and studying the first and second moments of the number (or more generally, the ``mass" with respect to some measure on $\Gamma$) of paths in $\Gamma$ which are contained in $T_k(G)$. This requires one to deal only with dependencies between vertices along  at most two paths.

A substantial amount of ``paths" techniques  were developed in the study of percolation. The most relevant ones to this work are the second moment method and the \textbf{EIT}\footnote{Exponential intersection tail.} property. 

In Section
\ref{sec: trees} we study $T_3$ on trees. As previously described, using
paths techniques, we reduce the complexity of dependencies we have to deal
with. Exploiting some hidden structure allows us to control the dependencies
along two paths (a similar
hidden structure is exploited also in the proof of Theorem \ref{thm: Z^dintro}). Essentially, we show that for every tree
$T$ as in Theorem \ref{thm: mainresulttreesinto} we have that $T_3(T)$ stochastically
dominates some super-critical \emph{quasi-independent} percolation process
on $T$ (see (\ref{eq:
QB}) for a precise statement).  Unfortunately, the situation is  more involved than that and we cannot use results regarding quasi-independent percolation as a black box. Instead, we perform a weighted second moment calculation.

The study of \emph{quasi-independent}
percolation on trees\footnote{Also known as quasi-Bernoulli percolation.} and its
relations to independent percolation is due to Lyons, \cite{lyons1989ising,lyons1992random}
  (see also Sections $5.3$-$5.4$ in \cite{lyons2005probability}).

The EIT method
was introduced in $\cite{cox1983oriented}$, where it was exploited for showing
that the critical value for oriented independent bond percolation on $\Z^d$
is $d^{-1}+O(d^{-3})$. The method was further extended in \cite{benjamini1998unpredictable}. The novelty of our use of the EIT method is that we apply it in a dependent setup using an auxiliary Markov chain which represents the different type of dependencies between two vertices. This idea can be utilized in some other dependent percolation models in which there is a bounded range of dependencies. 
 
\subsection{Notation and terminology}
Given a graph $G=(V,E)$, the connected component containing $v$ is denoted
by $C(v)$. We define the length of a path $\gamma$ as the number of edges it contains and denote it by $|\gamma|$. By abuse of notation we often identify
a path $\gamma=(\gamma_{1},\gamma_2,\ldots,\gamma_{k})$ with the set $\{\gamma_{1},\gamma_2,\ldots,\gamma_{k}
\}$. Throughout, we denote the $i$th vertex in $\gamma$ by $\gamma_i$ ($1 \le i \le |\gamma|+1$). We say that the path $\gamma$ is simple (or self-avoiding) if $\gamma_i \neq \gamma_j$, for all $i \neq j$. We denote the collection of all simple paths of length $k$ starting from $v \in V$ by $\Gamma_{v,k}$ and the collection of all such infinite simple paths by $\Gamma_{v}$.

Given a pair of vertices $u,v \in V$ their distance  $d_{G}(u,v)$ is defined as the length of the shortest path in $G$ starting at
$u$ and ending at $v$\footnote{If $u$ and $v$ are not connected by any path, then their distance is defined to be $\infty$.}. For $A,B \subseteq V$, the distance of $A$ from $B$ in $G$ is defined to be $d_{G}(A,B):=\inf \{ d_{G}(a,b) : a \in A,b \in B \}$. When $A$ and $ B$ are not disjoint, we define $d_{G}(A,B)=0$. When the underlining graph is clear from context, we simply write $d(A,B)$ and $d(u,v)$. When $A=\{u\}$ we write $d(u,B)$ instead of $d(\{u\},B)$.

For a pair of  vertices $u,v$, we write $u \sim v$ if $d(u,v)=1$ and say that $u$ is \emph{adjacent} to (or a \emph{neighbor} of) $v$. We say that a graph $G=(V,E)$ is \emph{locally finite} if $d_v < \infty$ for all $v \in V$. We say that a vertex $v$ is a \emph{neighbor} of a set $A$ and write $v \sim A$ (or of a path $\gamma$ and write $v \sim \gamma$), if $d(v,A)=1$ (respectively, $d(v,\gamma)=1$). That is, if $v \notin A$ (respectively, $v \notin \gamma$) and there is some $u \in A$ (respectively, $u \in \gamma$) such that $v \sim u$. If $d(v,A)>1$ we write $v \nsim A$. When $A=\{u\}$ we write $v \nsim u$.

We often abbreviate and write $\mathrm{w.p.}$ for ``with probability". We call a random variable which takes the value $1$  $\mathrm{w.p.}$  $p$ and the value $0$  $\mathrm{w.p.}$ $1-p$  a Bernoulli$(p)$ random variable. We denote the indicator of an event $A$ by $1_{A}$.

Several results that are used in the paper are quoted in the appendix (Section
\ref{sec: A}). All of these results are standard, with the exception of Theorem
\ref{thm: affiliation}.
\section{Infinite trees}
\label{sec: trees}
In this section we study $T_3$ on infinite trees and prove Theorem \ref{thm: mainresulttreesinto}.
Fix some simple path $\gamma$ of length $2k-1$ for some $k \ge 1$. 
 We shall dominate the restriction of $T_3$ to $\gamma$ using an auxiliary percolation process which is amenable to relatively neat analysis. For this purpose we partition $\gamma$ into successive pairs,  $\{\gamma_{2i-1},\gamma_{2i} \}$, $i=1,2,\dots,k$.  For each $i$ we define a certain ``good" event $A_i(\gamma)$, depending only on the ages of $ \gamma_{2i-1},\gamma_{2i}
$ and of that of their neighbors which do not lie in $\gamma$. The motivation behind the definitions shall be explain soon.

Let $T=(V,E,o)$ be a locally finite rooted tree. For any  $k \in \N$, let
$\Gamma_{2k-1}$ be the collection
of all self-avoiding paths of length $2k-1$ starting
at $o$. 

\begin{definition}
\label{def: goodpathonatreedef}
Let $T=(V,E,o)$ be a locally finite rooted tree. For any $v \in V \setminus \{o\}$ the parent of $v$, denoted by $\bar v $, is the unique vertex such that $\bar v \sim v $ and $d(o,\bar v)=d(o,v)-1$. Let $\ell_{r}:=\{v \in V:d(o,v)=r
\}$, $r \ge 0$.  Let $\gamma \in \Gamma_{2k-1}$. We define $N_{i}(\gamma):=\{u \sim \gamma_i : u \notin \gamma \}$ to be the set of neighbors of $\gamma_i$ which do not lie in $\gamma$. For every $i \in [k]$ and $j \in[2k]$ we define
\begin{equation*}
\begin{split}
M_{2i-1}(\gamma)&:=|\{u \in N_{2i-1}(\gamma) \cup \{\gamma_{2i} \} : X_{\gamma_{2i-1}}>X_u\}|, \\ M_{2i}(\gamma)&:=|\{u \in N_{2i}(\gamma) \cup \{\gamma_{2i-1} \} : X_{\gamma_{2i}}>X_u\}|,
\\ K_{j}(\gamma)&:=|\{u \in N_{j}(\gamma)  : X_{\gamma_{j}}>X_u\}|.
\end{split}
\end{equation*}
For all $1 < i < k$, let $A_{i}(\gamma)$ be the event that $\max ( M_{2i-1}(\gamma), M_{2i}(\gamma)
 ) \le 1$.  Let $A_{1}(\gamma)$ (resp.~$A_{k}(\gamma)$) be the event that $K_{1}(\gamma) \le 2$, $K_2(\gamma) \le 1$ (resp.~$K_{2k}(\gamma) \le 2$, $K_{2k-1}(\gamma) \le 1$).  We say that a path $\gamma \in \Gamma_{2k-1} $ is \emph{good}, if the event $A_{\gamma}:=\bigcap_{i=1}^k A_{i}(\gamma)$ occurs.  For any path $\gamma \in \Gamma_{k}$ we set
\begin{equation}
\label{Y'gamma}
w(\gamma):= \frac{1}{d_o}\prod_{i=2}^{2k-1}
\frac{1}{d_{\gamma_i}-1} \text{ and } Y_{\gamma}:=\frac{w(\gamma)1_{A_{\gamma}}}{\Pr[A_{\gamma}]}.
\end{equation}
Finally, we say that $o$ is $ k$\emph{-good},
if $\sum_{\gamma \in \Gamma_{2k-1}}Y_{\gamma} \ge \frac{1}{2}$. 
\end{definition} 
Note that by construction, a \emph{good} path must be contained in $T_3$. The idea
behind the definition of a \emph{good} path is that for a fixed $\gamma \in \Gamma_{2k-1}$, the
events $A_1(\gamma),A_2(\gamma),\ldots,A_{k}(\gamma)$ are mutually independent, as they
depend on ages of disjoint sets of vertices (since $T$ is a tree). This avoids dealing with the accumulation of dependencies, which may occur when working with the indicators of the vertices along a certain path belonging to $T_3$.

Fix some path  $\gamma \in \Gamma_{2k-1} $. The analysis would have been much simpler if instead of considering  $A_1(\gamma),A_2(\gamma),\ldots,A_{k}(\gamma)$  we could work
with the events $I_1(\gamma),\ldots,I_{2k}$, where $I_i(\gamma)$ is
the event that $\gamma_i$ is younger than all of its neighbors which do not lie in $\gamma$.
The events $I_{1}(\gamma),\ldots,I_{2k}(\gamma)$
 are independent and for all $i$ we have that $I_{i}(\gamma)$
is contained in the event that $\gamma_i \in T_3$. However,  the $I_i(\gamma)$'s give rise to a critical percolation process
which does not suffice for our purposes.

 We now describe a growth condition that shall be assumed throughout the section. Whenever the minimal degree of $T$ is strictly greater than 2, the condition takes a simple form and coincides with the growth condition appearing in Theorem \ref{thm: mainresulttreesinto}. In particular, our analysis implies the assertion of Theorem \ref{thm: mainresulttreesinto}, but allows us to deal also with vertices of degree 2.

For any $v \in T$ let $\gamma^{(v)}$ be the unique self-avoiding path in $T$ which starts at $o$ and ends at $v$. We define $q_{v}:=|\{1<i < \lfloor  \frac{d(v,o)+1}{2}  \rfloor : \max (d_{\gamma_{2i-1}^{(v)}},d_{\gamma_{2i}^{(v)}})>2 \}|$. Consider the
following condition. There exist $C>0$ and $1 \le a < 4/3 $, such that 
\begin{equation}
\label{eq: growthcondition}
d_{v}-1 \le Ca^{q_{v}} d_o, \text{ for all }v \in V.
\end{equation}
\begin{definition}
For any distinct paths $\gamma , \gamma'\in
\Gamma_{2k-1}$, let  $\gamma \wedge \gamma'$ be the
 furthest vertex from $o$ belonging to both paths. Let $|\gamma \wedge
\gamma'|:=d(o,\gamma \wedge\gamma')+1$ (so $\gamma_{|\gamma \wedge
\gamma'|}=\gamma \wedge \gamma'$). Finally, we define $$\gamma \cap \gamma':=(\gamma_1,\gamma_2,\ldots,\gamma_{|\gamma \wedge
\gamma'|})
\in \Gamma_{|\gamma \wedge
\gamma'|}.$$
\end{definition}
Essentially, we will show that the percolation process defined via the \emph{good}-paths is a quasi-(independent) Bernoulli percolation
process (see $\cite{lyons2005probability}$  section $5.4$ for a precise definition and details). More precisely, we show that there exists an absolute constant $M>0$  such that 
\begin{equation}
\label{eq: QB}
\begin{split}
& \Pr[A_{\gamma} \cap A_{\gamma'}] \ge \frac{\Pr[A_{\gamma}]\Pr[A_{\gamma'}]}{M\Pr[A_{\gamma \cap \gamma'}]}, \text{ for all } \gamma,\gamma' \in \Gamma_{2k-1} \text{ such that }|\gamma \wedge
\gamma'| \text{ is even},
\\ & \Pr[A_{\gamma} \cap A_{\gamma'}] \ge \frac{\Pr[A_{\gamma}]\Pr[A_{\gamma'}]d_{|\gamma \wedge
\gamma'|}}{M\Pr[A_{(\gamma
\cap \gamma' \setminus \{\gamma \wedge
\gamma' \})}]}, \text{ for all } \gamma,\gamma' \in \Gamma_{2k-1} \text{ such
that }|\gamma \wedge
\gamma'| \text{ is odd}.
\end{split}
\end{equation}
Although (\ref{eq: QB}) bears a resemblance to the condition which defines a quasi-Bernoulli percolation  process, because the $A_i(\gamma)$'s are defined in terms of pairs
of vertices, we cannot use results concerning quasi-Bernoulli percolation as a black box. 

Nevertheless, such results are useful in terms of intuition. Moreover, since the main tool utilized in the development of the theory of quasi-Bernoulli percolation is the second moment method, it is only natural to utilize the same technique when attempting to transform our intuition into a rigorous proof.

 As can be seen from (\ref{eq: marginalsforgoodpath'}) and (\ref{eq: firstrowof2.2}) below, in conjunction with the independence of the events  $A_1(\gamma),A_2(\gamma),\ldots,A_{k}(\gamma) $, for each fixed $\gamma \in \Gamma_{2k-1} $, the marginal probabilities of the $A_\gamma$'s  would have implied super-criticality for independent site percolation, at least when $T$ is of bounded degree. Thus, if we actually had a quasi-Bernoulli percolation process in hand, then Theorems 5.19 and 5.24 in $\cite{lyons2005probability}$ (originally from $\cite{lyons1989ising, lyons1992random}$) would have implied that $T_3(T)$ is super-critical (at least when $T$ is of bounded degree).
\begin{theorem}
\label{thm: T3fortrees}
 Let $T=(V,E,o) $ be a rooted tree with no leaves. Assume that condition (\ref{eq: growthcondition}) holds for some $C>0$ and $1<a <4/3$. Then there exists some constant $c_{1}>0$, such that $o$ is $ k$-good with probability at least $c_{1}/d_{o}$,  for all $k>1$.  
\end{theorem}
\begin{proof}
 Fix some $k >1$. For a path $\gamma \in \Gamma_{2k-1}$ let $Y_{\gamma}$
be as in (\ref{Y'gamma}). Then $\mathbf{E}(Y_{\gamma}) = w(\gamma),
\text{ for
all } \gamma \in \Gamma_{2k-1}$. We define $Z_{k}:=\sum_{\gamma \in \Gamma_{2k-1}}
Y_{\gamma}$. By the definition of $w(\gamma)$ we have that $\sum_{\gamma
\in \Gamma_{2k-1}}w(\gamma)=1$. Hence, $\mathbf{E}[Z_{k}] =1$.
By the Paley-Zygmund inequality, \[\Pr(o \text{ is } k\text{-good})= \Pr(Z_k>1/2) \ge \frac{1}{4\mathbf{E}[Z_k^2]}.\] Thus our task is to estimate $\mathbf{E}[Z_k^2]$ from above.
Note that
\begin{equation}
\label{eq: Zksquaretrees}
\mathbf{E}[Z_k^2]=\sum_{\gamma\in \Gamma_{2k-1}}\sum_{\gamma' \in \Gamma_{2k-1}}\mathbf{E}[Y_{\gamma}
 Y_{\gamma'} ]=\sum_{\gamma\in \Gamma_{2k-1}}\sum_{\gamma' \in \Gamma_{2k-1}}w(\gamma)w(\gamma')\frac{\Pr[ A_{\gamma} \cap A_{\gamma'}]}{ \Pr[ A_{\gamma}]\Pr[A_{\gamma'}]}.
\end{equation}
Let $\gamma\in \Gamma_{2k-1}$. We now calculate $\Pr[A_{\gamma}]$.  Since we are considering
a tree, by construction, the events $A_1(\gamma),A_2(\gamma),\ldots,A_{k}(\gamma) $ depend on ages of disjoint sets of vertices. Hence they are
independent. Denote $r_i(\gamma):=1_{d_{\gamma_{2i-1}}>2}+1_{d_{\gamma_{2i}}>2}$.  Fix some $i \neq 1,k$.  There are three cases in which we can have $\{M_{2i-1}(\gamma) \le 1, M_{2i}(\gamma) \le 1 \}$: \begin{itemize} \item[Case 1:]  $K_{2i-1}(\gamma)=0=K_{2i}(\gamma)$ - That is, both $\gamma_{2i-1}$ and $\gamma_{2i}$ are younger than all of their neighbors (resp.) which do not lie in $\gamma$. Because $N_{2i-1}(\gamma) \cap N_{2i}(\gamma)=\emptyset $, we have that $K_{2i-1}(\gamma)$ and $K_{2i}(\gamma)$ are independent and so  the probability they both equal 0 is $\frac{1}{(d_{\gamma_{2i-1}}-1)(d_{\gamma_{2i}}-1)} $.   \item[Case 2:] $K_{2i-1}(\gamma)=1, X_{\gamma_{2i}}>X_{\gamma_{2i-1}},K_{2i}(\gamma)=0$ - That is,  $\gamma_{2i}$ is younger than all of its neighbors which do not lie in $\gamma$, but older than $\gamma_{2i-1}$, while  $\gamma_{2i-1}$ has exactly one younger neighbor not lying in $\gamma$. There are $d_{\gamma_{2i-1}}-2 $ possibilities for the identity of the youngest member of $N_{2i-1}(\gamma)  $, which in fact has to be the youngest in $N_{2i-1}(\gamma) \cup N_{2i}(\gamma) \cup \{\gamma_{2i-1},\gamma_{2i} \}$ (contributing a $\frac{d_{\gamma_{2i-1}}-2}{d_{\gamma_{2i-1}}+d_{\gamma_{2i}}-2}$ factor), while $\gamma_{2i-1} $ has to be the second youngest in this set (contributing a $\frac{1}{d_{\gamma_{2i-1}}+d_{\gamma_{2i}}-3}$ factor). Conditioned on that, the distribution of the relative order between $N_{2i}(\gamma) \cup \{\gamma_{2i} \} $ is still uniform over all orderings, and so the conditional probability that $\gamma_{2i}$ is the youngest in this set is $\frac{1}{d_{\gamma_{2i}}-1} $. Overall, the probability of Case 2 is $\frac{d_{\gamma_{2i-1}}-2}{(d_{\gamma_{2i-1}}+d_{\gamma_{2i}}-2)(d_{\gamma_{2i-1}}+d_{\gamma_{2i}}
-3)(d_{\gamma_{2i}}-1)} $.
  \item[Case 3:] $K_{2i-1}(\gamma)=0, X_{\gamma_{2i}}<X_{\gamma_{2i-1}},K_{2i}(\gamma)=1$ (same as Case 2, with the roles of $2i-1$ and $2i$ exchanged), which has probability $\frac{d_{\gamma_{2i}}-2}{(d_{\gamma_{2i-1}}+d_{\gamma_{2i}}-2)(d_{\gamma_{2i-1}}+d_{\gamma_{2i}}
-3)(d_{\gamma_{2i-1}}-1)} $.
\end{itemize}
 We get that:
\begin{equation}
\label{eq: marginalsforgoodpath'}
\Pr[ A_i(\gamma)] = \begin{cases}
 \frac{1}{(d_{\gamma_{2i-1}}-1)(d_{\gamma_{2i}}-1)}+\frac{1}{(d_{\gamma_{2i-1}}+d_{\gamma_{2i}}-2)(d_{\gamma_{2i-1}}+d_{\gamma_{2i}}
-3)} \left[\frac{d_{\gamma_{2i-1}}-2}{d_{\gamma_{2i}}-1}+\frac{d_{\gamma_{2i}}-2}{d_{\gamma_{2i-1}}-1} \right]
, & 1<i< k,  \\
\frac{2}{d_{o}(d_{\gamma_{2}}-1)}, & i=1, \\
\frac{2}{(d_{\gamma_{2k-1}}-1)d_{\gamma_{2k}}},    & i=k.
\end{cases}
\end{equation}
\begin{claim}
\label{claim:minat33}
The function $f(x,y)=\frac{(x-1)(x-2)+(y-1)(y-2)}{(x+y-2)(x+y-3)} $ attains its minimum in the domain $D:=\{(x,y) \in \R^2:x,y \ge 3 \} $ at $(x,y)=(3,3)$ and $f(3,3)=1/3$. Whereas $f(x,2)=1-\frac{2}{x}$ and $f(2,y)=1-\frac{2}{y}$.
\end{claim}
\begin{proof}
Observe that since $(x+y)^2 \le 2 x^2+2y^2$, we have that  $$\limsup_{r \to \infty} \sup_{(x,y)\in D:x+y \ge r} f(x,y) =\limsup_{r \to \infty} \sup_{(x,y)\in D:x+y \ge r}\frac{x^2+y^2}{(x+y)^2}=1/2$$   Since  $f(x,y)<1/3<1/2 $ this means that $f$ has to attain a global minimum in the domain  $D$.   
We now verify that there is no solution to $\partial_x f(x,y)=0=\partial_y f(x,y) $ in $D$. Indeed,
$\partial_x f(x,y)=\frac{(y-1)(x^{2}-2x-y^{2}+3y-1)}{(x+y-2)^{2}(x+y-3)^{2}} $ whose only root which is greater or equal to 3 is $x_{0}(y)=1+\sqrt{(y-1)(y-2)}$. By symmetry, if $\partial_y f(x,y)=0$, then $y=y_0(x)=1+\sqrt{(x-1)(x-2)} $. Hence if $\partial_x f(x,y)=0=\partial_y f(x,y) $, we must have that \[x=1+\sqrt{(y-1)(y-2)} =1+\sqrt{(\sqrt{(x-1)(x-2)})(\sqrt{(x-1)(x-2)}-1)},\]
which implies that $x=\sqrt{(x-1)(x-2)}$, a contradiction!  Thus if the minimum is attained at $(x,y)$, we  must have that either $x=3$ or $y=3$ (since the minimum must be attained at the boundary of the domain). By symmetry, we may assume that $y=3$. Minimizing $f(x,3)$ over $x \ge 3$, it is not hard to verify that the minimum is attained at $x=3$. 
\end{proof}
Note that if $1<i< k $ and $\max (d_{\gamma_{2i-1}},d_{\gamma_{2i}})=2$ (i.e.~$r_i(\gamma)=0 $) we get that $\Pr[A_{i}(\gamma)]=1$, while if $r_i(\gamma)>0$ (i.e.~$\max (d_{\gamma_{2i-1}},d_{\gamma_{2i}})>2 $) then by  \eqref{eq: marginalsforgoodpath'} and Claim \ref{claim:minat33}
\begin{equation}
\label{eq: firstrowof2.2}
\Pr[ A_i(\gamma)] \ge \frac{4}{3(d_{\gamma_{2i-1}}-1)(d_{\gamma_{2i}}-1)}, \text{ for all } 1<i< k \text{ such that } r_i(\gamma)>0.
\end{equation}

We now turn to the task of finding an upper bound on $\Pr[A_{\gamma}\cap A_{\gamma'}]$ for $\gamma,\gamma' \in \Gamma_{2k-1}$. Fix some distinct  $\gamma, \gamma' \in \Gamma_k$.  If $|\gamma \wedge \gamma'| \in \{2i-1,2i\}$, we define $B_{\gamma,\gamma'}$
to be the event that for $j \in \{2i-1,2i\}$ we have that  $|\{u : X_{\gamma_{j}}>X_u, u \sim \gamma_j,
u \notin \gamma \cup \gamma' \}| \le 1 $ and that also  $|\{u : X_{\gamma'_{j}}>X_u, u \sim \gamma'_j,
u \notin \gamma \cup \gamma' \}| \le 1$.

Note that $A_i(\gamma)\cap A_i(\gamma')
\subset B_{\gamma,\gamma'}$. Thus if $|\gamma
\wedge \gamma'|=j<2k$, then 
\begin{equation}
\label{eq: secondmomenttree3'}
A_{\gamma} \cap A_{\gamma'}  \subset \left(\bigcap_{i \in [k] \setminus \{\lfloor j/2 \rfloor \}}
A_{i}(\gamma) \right) \cap B_{\gamma,\gamma'}  \cap
\left(\bigcap_{i=\lfloor j/2 \rfloor +1}^{k}
A_{i}(\gamma') \right), 
\end{equation}
where if $|\gamma
\wedge \gamma'|=2k-2$, the rightmost intersection does not appear. Moreover, from the
definition of $B_{\gamma,\gamma'}$, we have that the $2k- \lfloor j/2 \rfloor
$ events appearing in the right hand side of (\ref{eq: secondmomenttree3'})
are jointly
independent, as they depend on ages of disjoint sets of vertices. Consequently,
if $|\gamma
\wedge \gamma'|=j<2k$ then\begin{equation}
\label{eq: secondmomenttree4'}
\frac{\Pr[ A_{\gamma} \cap A_{\gamma'}]}{ \Pr[ A_{\gamma}]\Pr[A_{\gamma'}]}
\le \frac{\Pr[ B_{\gamma,\gamma'}]}{\Pr[ A_{\lfloor
j/2 \rfloor}(\gamma)]\Pr[ A_{\lfloor j/2 \rfloor}(\gamma')]\prod_{0 \le i
< \lfloor
j/2 \rfloor}\Pr[A_{i}(\gamma)]}  .
\end{equation}
We now calculate $\Pr[B_{\gamma,\gamma'}]$. Assume that $|\gamma \wedge \gamma'| \in \{2i-1,2i \}$. If $|\gamma
\wedge \gamma'|=2i$, then $B_{\gamma,\gamma'}$ consists of two independent events, defined in terms of $\gamma_{_{2i-1}}$ and $\gamma_{_{2i}}$, respectively. Similarly, if  $|\gamma
\wedge \gamma'|=2i-1$, then $B_{\gamma,\gamma'}$ consists of three independent
events, defined in terms of $\gamma_{_{2i-1}}$, $\gamma_{_{2i}}$ and $\gamma_{_{2i}}'$, respectively. A simple calculation yields that:
\begin{equation}
\label{eq: B(gamma,gamma')'}
\Pr[ B_{\gamma,\gamma'}] = \begin{cases}
8(d_{ \gamma_{_{2i}} }-1)^{-1}(d_{\gamma'_{2i}}-1)^{-1}(d_{ \gamma_{_{2i-1}} }-2)^{-1}, & |\gamma \wedge \gamma'|=2i-1, i \notin \{1,2k-1\},  \\
4(d_{\gamma_{2i}}-2)^{-1}(d_{ \gamma_{_{2i-1}} }-1)^{-1}, & |\gamma \wedge \gamma'|=2i, i \notin \{1,2k-1\}, \\
8(d_o-1)^{-1}(d_{\gamma_{_{2}}}-1)^{-1}(d_{\gamma'_{_{2}}}-1)^{-1}, &  |\gamma \wedge
\gamma'|=1, \\
4(d_o (d_{\gamma_{_{2}}}-2))^{-1}, & |\gamma \wedge
\gamma'|=2, \\
8((d_{\gamma_{2k-1}}-2)d_{\gamma_{2k}}d_{\gamma'_{2k}})^{-1}, &  |\gamma \wedge
\gamma'|=2k-1. \\
\end{cases}
\end{equation}
Plugging this and  (\ref{eq: marginalsforgoodpath'})-(\ref{eq: firstrowof2.2}) in (\ref{eq: secondmomenttree4'})
yields that if $|\gamma
\wedge \gamma'|=j<2k$, then
\begin{equation}
\label{eq: secondmomenttree4''}
\frac{\Pr[ A_{\gamma} \cap A_{\gamma'}]}{ \Pr[ A_{\gamma}]\Pr[A_{\gamma'}]} \le 8 \cdot (3/4)^{|\{1< m < \lfloor j/2 \rfloor  :r_{m}(\gamma)\neq 0 \}|}(d_{o}/2)\prod_{1
\le i \le j}(d_{\gamma_i}-1).
\end{equation}
Denote $W(v):= \sum_{\tilde \gamma \in \Gamma_k : v \in \tilde
\gamma }w(\tilde \gamma)$ for every $v \in V$ such that $d(v,o) \le 2k-1$. Note
that
\begin{equation}
\label{eq: Wgammaj}
d_o  \prod_{2
\le i \le j-1}(d_{\gamma_i}-1)=\frac{1}{W(\gamma_{j})}, \text{ for all } \gamma \in \Gamma_{2k-1} \text{ and } j
\le 2k.
\end{equation}
 Let $1 \le a < 4/3$ and $C>0$ be as in condition (\ref{eq: growthcondition}). Then $\alpha:=\frac{3a}{4}<1$.
 Recall that
from the definition of $q_v$, from condition (\ref{eq: growthcondition}),
$q_{\gamma_{j}}=|\{1< m < \lfloor j/2 \rfloor    :r_{m}(\gamma)\neq 0\}|$. By condition (\ref{eq: growthcondition}), $d_{\gamma_j}-1 \le Ca^{q_{\gamma_{j}}} $. Hence  $$4 \cdot (d_{\gamma_j}-1)(3/4)^{q_{\gamma_{j}}} \le C_{1} \alpha^{q_{\gamma_{j}}} d_o,$$ This, together with (\ref{eq: secondmomenttree4''})-(\ref{eq:
Wgammaj}), implies that
\begin{equation}
\label{eq: BcomparedtoAcupA}
\begin{split}
& \frac{\Pr[ A_{\gamma} \cap A_{\gamma'}]}{ \Pr[ A_{\gamma}]\Pr[A_{\gamma'}]}\le \frac{C_{1}\alpha^{q_{\gamma_{j}}} d_{o}}{W(\gamma_{j})}, \text{ for all }\gamma,\gamma' \in \Gamma_{2k-1} \text{ such that }|\gamma \wedge \gamma'|=j<2k.
\end{split} 
\end{equation}

It is easy to verify that $\frac{\Pr[ A_{\gamma} \cap A_{\gamma'}]}{ \Pr[ A_{\gamma}]\Pr[A_{\gamma'}]} \le \frac{C_{1}\alpha^{q_{\gamma_{2k}}} d_{o}}{W(\gamma_{2k} )} $ also when $\gamma=\gamma'$. 
 Note for all  $\bar \gamma \in \Gamma_{2k-1}$ we have that $\sum_{\tilde \gamma \in \Gamma_{2k-1} : |\tilde \gamma  \wedge \bar \gamma'|=r} w(\tilde \gamma) \le W(\bar \gamma_{r} )$, for all $r \le 2k$. Hence, by (\ref{eq: BcomparedtoAcupA}) we have that:
\begin{equation}
\label{eq: secondmomenttree5}
\begin{split}
& \sum_{\gamma' \in \Gamma_{2k-1}}\mathbf{E}[Y_{\gamma}  Y_{\gamma'} ]\le C_{1}d_{o}w(\gamma)_{}^{}\sum_{j=1}^{2k} \sum_{\gamma' \in \Gamma_{2k-1} : |\gamma \wedge \gamma'|=j}
w(\gamma')\alpha^{q_{\gamma_{j}}}/W(\gamma_j)
\\ & \le C_{1}d_{o}w(\gamma)[2+ \sum_{3 \le j \le2 k : d_{\gamma_j}>2}
\sqrt{\alpha}^{|\{3\le i \le j-2 : d_{\gamma_{i}}>2\}|}] \le Md_{o} w(\gamma),
\end{split}
\end{equation}
for some constant $M>0$. Hence 
\[\mathbf{E}[Z_k^2]=\sum_{\gamma\in \Gamma_{2k-1}}\sum_{\gamma' \in \Gamma_{2k-1}}\mathbf{E}[Y_{\gamma}  Y_{\gamma'} ]\le\sum_{\gamma\in \Gamma_{2k-1}} M d_{o}w(\gamma)= Md_{o}.\]
Finally, by the Paley-Zygmund inequality, $\Pr(Z_k>1/2) \ge \frac{1 }{4Md_{o}}$.
\end{proof}
\begin{lemma}
\label{lem: 0-1infinitecluster}
Let $G=(V,E)$ be an infinite connected graph. Let $\ell \in \N$. For any $v \in V$, let $Y_v$ be the indicator of the event that $v \in T_{\ell}(G)$. Then the tail $\sigma$-algebra of $(Y_{v})_{v \in V}$ is trivial.  Consequently, $\Pr[A_{\ell}] \in \{0,1\} $, where $A_{\ell}$ is
the event that $T_{\ell}(G)$ has an infinite cluster.
\end{lemma}
\begin{proof}
  One can readily verify that the tail $\sigma$-algebra of $(Y_v)_{v \in V}$ is contained in the tail $\sigma$-algebra of $(X_v)_{v \in V}$. Thus it is trivial by Kolmogorov's 0-1 law. It is easy to see that $A_{\ell}$ is in the tail $\sigma$-algebra of $(Y_v)_{v \in V}$.
\end{proof}
\begin{corollary}
\label{thm: mainresulttrees}
Let $T=(V,E,o) $ be a rooted tree with no leaves. Assume that condition
(\ref{eq: growthcondition}) holds. Then
$\Pr[|C_o(T_3)|=\infty] \ge c_{1}/ d_{o}$, where $C_o(T_3)$ is the connected component of $o$ in $T_3$ and $c_1>0$ is as Theorem \ref{thm: T3fortrees}. Consequently, $T_3$ has an infinite cluster $\mathrm{a.s.}$
\end{corollary}
\begin{proof}
The event $|C_o(T_3)|=\infty$ contains the decreasing intersection $\bigcap_{ k \ge 2 }\{Z_k>0\}$. So by Theorem \ref{thm: T3fortrees}, $\Pr[|C_o(T_3)|=\infty] \ge c_{1} d_{o}^{-1}$. The proof is concluded using Lemma \ref{lem: 0-1infinitecluster}. \end{proof}
We end the section with a modification of Theorem \ref{thm: T3fortrees} which we shall need in Section \ref{sec: expanders}.
\begin{definition}
\label{def: auxdefforrandomgraphs}
Let $T=(V,E,o)$ be an infinite tree rooted tree with no leaves, satisfying
condition (\ref{eq: growthcondition}). Let  $I \subset V $ be such that $d(u,v)
\ge 15$, for all $u,v \in I$. Assume that $d_v \ge 3$, for all
$v \in V \setminus I$. Let $k \ge 15$ be an odd integer. We say that a path
$\gamma \in \Gamma_k$
is \emph{nice} if it is \emph{good} (i.e.~in the notation of Definition \ref{def:
goodpathonatreedef}, the event $A_{\gamma}=\bigcap_{0 \le i \le  \frac{k-1}{2}}A_i(\gamma)$
occurs) and $\gamma \cap
I \subset T_2$. Let $W_{o,k}$ be the union of all \emph{nice} paths in
$\Gamma_k$.
\end{definition}
\begin{lemma}
\label{lem: auxforsec6}
 There exist some $b>1$ and $c_{2}>0$ (both independent of $k$) such that
\begin{equation}
\label{opengoodpaths}
\Pr \left[\left| W_{o,k} \right| >b^{k} \right]>c_2/d_{0}, \text{ for all odd }k \ge 15.
\end{equation}
\end{lemma}
\begin{proof} Let $\gamma=(\gamma_0,\ldots,\gamma_k) \in \Gamma_k$. Let $C_{\gamma}$ be the event that $\gamma$ is \emph{nice}. For any $0 \le i \le \frac{k-1}{2}$ such that $I \cap \{\gamma_{2i},\gamma_{2i+1}\}$ is non-empty, set $f_{i}=2i$ and $g_i=2i+1$ if $\gamma_{2i} \in I$ and set $f_{i}=2i+1$ and $g_i=2i$
if $\gamma_{2i+1} \in I$.  For $0 \le i \le \frac{k-1}{2}$ let $D_{i}(\gamma)=A_{i}(\gamma)$ if $I \cap \{\gamma_{2i},\gamma_{2i+1}\}$ is empty. Otherwise, set $D_{i}(\gamma)$ be the event that $\gamma_{f_i}$ and $\gamma_{g_{i}}$ are both younger than all their neighbors not lying in $\gamma$ (resp.) and that $\gamma_{f_i}$ is younger than $\gamma_{g_{i}}$. A simple calculation, similar to (\ref{eq: marginalsforgoodpath'}), shows that $4\Pr[D_{i}(\gamma)] \ge \Pr[A_{\gamma}]$, for all $i$. Note that $D_{\gamma}:=\bigcap_{0 \le i \le  \frac{k-1}{2}}D_i(\gamma) \subset C_{\gamma}$, and that $D_0(\gamma),D_1(\gamma),\ldots,D_{\frac{k-1}{2}}(\gamma) $ are jointly independent. Consequently,
\begin{equation}
\label{eq: CgammavsAgamma}
\Pr[D_{\gamma}] \ge \Pr[A_{\gamma}]4^{-|A \cap \gamma|}.
\end{equation}
Similarly to (\ref{Y'gamma}), set $w(\gamma):=\left[ d_o\prod_{i=1}^{k-1}
(d_{v_i}-1)\right]^{-1}$, $\bar Y_{\gamma}:=\frac{w(\gamma)1_{D_{\gamma}}}{\Pr[D_{\gamma}]} $ and $\bar Z_k:=\sum_{\gamma \in \Gamma_k}\bar Y_{\gamma}$.

Let $\gamma,\gamma' \in \Gamma_k$ be such that $|\gamma \wedge \gamma'| \in \{2i,2i+1 \} $. If $F_{\gamma,\gamma'}$ be the event that $B_{\gamma,\gamma'} $ occurs and that if $\{\gamma_{2i},\gamma_{2i+1}\} \cap I $ or $\{\gamma'_{2i},\gamma'_{2i+1}\} \cap I $ are non-empty, then $\gamma_{f_i}$
is younger than $\gamma_{g_{i}}$ or $\gamma'_{f_i}$
is younger than $\gamma'_{g_{i}}$, respectively. Then $2\Pr[F_{\gamma,\gamma'}] \ge \Pr[B_{\gamma,\gamma'}] $ and similarly to (\ref{eq: secondmomenttree4'}), we have that
\begin{equation}
\label{eq: secondmomenttreemodified}
\frac{\Pr[ D_{\gamma} \cap D_{\gamma'}]}{ \Pr[ D_{\gamma}]\Pr[D_{\gamma'}]}
\le \frac{\Pr[F_{\gamma,\gamma'}]}{\Pr[ D_{i}(\gamma)]\Pr[ D_{i}(\gamma')]\prod_{0 \le r
<i}\Pr[D_{i}(\gamma)]}  .
\end{equation}
Similarly to the proof of Theorem \ref{thm: T3fortrees},  we have that
$$\mathbf{E}[\bar Z_k^2] \le d_0/4c_{2}, \text{ for some constant }c_{2}>0.$$ Hence by the Paley-Zygmund inequality (Theorem \ref{thm: payleyzygmund}), $\Pr[\bar Z_k>1/2]
\ge c_{2}d_0^{-1}$. Note that by (\ref{eq: CgammavsAgamma}),
  (\ref{eq: marginalsforgoodpath'}) and (\ref{eq: firstrowof2.2})
together with the independence of $A_0(\gamma),\ldots,A_{ \frac{k-1}{2}}(\gamma)$, we have $\frac{w(\gamma)}{\Pr[C_{\gamma}]}
\le b^{-k}/2$ for some constant $b>1$ independent of $k$, for any $\gamma \in \Gamma_k$ . Hence on the event $\bar Z_k > 1/2$, it must be the case that $\left| W_{o,k} \right| >b^{k}$.
\end{proof}
\section{The first 2 layers}
\label{sec: T2}
In this section we show that for any locally finite graph $G$ with a countable vertex set, all connected components of $T_2$ are $\mathrm{a.s.}$ finite. This demonstrates that although the marginal probabilities in $T_2$ would be super-critical in the independent setup,  for any infinite connected bounded degree tree,  the dependencies affect the global properties of $T_2$.
We start with a simple
lemma.

\begin{lemma}
\label{lem:monotone}
Given $G=(V,E)$ let $\gamma=(v_1,...,v_s)$ be a simple path in $G$. Suppose that $\{v_1,...,v_s\}\subseteq
T_2(G)$ and $\min(X_{v_1},...,{X_{v_s}})=X_{v_1}$.
Then the sequence $X_{v_1},...,{X_{v_s}}$ is monotonically increasing. Moreover, for all $2<i \in [s]$, $v_{i} \nsim \{v_{j} : j \in [i-2] \}$. Consequently, $T_2(G)$ is $\mathrm{a.s.}$ a forest. 
\end{lemma}
\begin{proof}
If the sequence was not increasing, there would be a vertex $v_{\ell} \in \gamma$ such that $X_{v_{\ell}} > X_{v_{\ell+1}}$
and $X_{v_{\ell}} > X_{v_{\ell-1}}$. But by the definition of $T_2$, it cannot be
the case that $v_{\ell}$ belongs to $T_2$, a contradiction! Now, if $2<v_{i} \sim v_{j}$ for some $j \in [i-2]$, then using the monotonicity which was just established, $X_{v_{j}}<X_{v_{i-1}}<X_{v_{i}}$, so it cannot be the case that $v_i \in T_2$. 
\end{proof}
As a warm-up, we first consider the case that $G=(V,E)$ is of bounded degree, as it is significantly simpler.
\begin{theorem}
\label{thm: T2isfinitebdddegreecase}
Let $G$ be a graph with a countable vertex set $V$ and $\max_{v \in V}d_v =: \Delta < \infty$.  Then, $\Pr[T_2(G) \text{ has an infinite connected component}]=0$.
\end{theorem}
\begin{proof}
We may assume w.l.o.g.~that $G$ is connected. By Lemma \ref{lem:monotone}, the event that $v \in V$ belongs to an infinite cluster of $T_2$ is equal to the
event that there exists some $\gamma \in \Gamma_{v,\infty}$  such that $X_{\gamma_{i}}<X_{\gamma_{i+1}}$, for all $i \in \N$ and that $\gamma \subset T_2$. Call
the previous event $I_{v,\infty}$.

We now show that $\Pr[I_{v,\infty}]=0$, for all $v \in V$. Fix some $v \in V$. For every $n \in \N$ we define $\Gamma_{v,n}'$ to be the collection of all   $\gamma  \in \Gamma_{v,n}$, such that
$d_{\gamma_{\ell}} \ge 2$ and $\gamma_i \nsim \gamma_j$, for all $2 \le \ell \le n+1$ and all  $1\le j \le i-2 \le n-1$. For every $\gamma \in \Gamma'_{v,n}$ we define $L_{\gamma}$ to be the
event that $\gamma \subset T_2$ and that $v$ is the youngest vertex in $\gamma$.
  Define $I_{v,n}:= \bigcup_{\gamma
\in \Gamma'_{v,n}}  L_{\gamma}$. By Lemma \ref{lem:monotone}
 
\begin{equation}
\label{eq: 3.3}
\Pr[L_{\gamma}] \le \frac{1}{(n+1)!}, \text{ for all }\gamma
\in \Gamma'_{v,n}, \text{ for all }n.
\end{equation}
Clearly, $|\Gamma'_{v,n}| \le \Delta^{n}$. Hence by (\ref{eq: 3.3}) and a union bound over all $\gamma \in \Gamma'_{v,n}$, we get that $$\Pr[I_{v,n}] \le \frac{\Delta^n}{(n+1)!} \to 0, \text{ as }n \to \infty.$$ By
Lemma \ref{lem:monotone}, the decreasing intersection $\bigcap_{n \ge 1} I_{v,n}$ equals $I_{v,\infty}$. So $\Pr[I_{v,\infty}]=0$, for all $v \in V$. We are done, since
$\Pr[T_2 \text{ has an infinite cluster}] \le \sum_{v \in V}\Pr[I_{v,\infty}]=0$.
\end{proof}
\begin{theorem}
\label{thm: T2isfinitenonbdddegreecase}
Let $G=(V,E)$ be a locally finite graph with a countable vertex set $V$.  Then, $\Pr[T_2(G) \text{ has an infinite connected component}]=0$.
\end{theorem}
\begin{proof}
For all $v \in V$ and $n \in \N$ we define $I_{v,n}$ and $\Gamma'_{v,n}$ as in the proof of Theorem \ref{thm: T2isfinitebdddegreecase}. As in the proof of Theorem \ref{thm: T2isfinitebdddegreecase}, it suffices to show that $\Pr[I_{v,n}] \to 0$, as $n \to \infty$, for
all $v \in V$. Instead of a straightforward union bound over all $\gamma \in \Gamma'_{v,n}$, which is difficult to perform in the non-bounded degree setup, we perform a weighted first moment calculation which gives rise to a recurrence relation with respect to $n$. 

Fix some $v \in V$ with $d_{v}>1$. For every $n>1$, we define $$\kappa(\gamma):= \prod_{i=2}^{n} \left( \frac{d_{\gamma_{i}}} {d_{\gamma_{i}}-1} \right) \left( \frac{d_{\gamma_{i}}+d_{\gamma_{i-1}}-1}{d_{\gamma_{i-1}}}  \right), \text{ for every }\gamma:=(v=\gamma_{1},\ldots,\gamma_{n+1})
\in \Gamma'_{v,n}.$$ We will show that
\begin{equation}
\label{eq: wtoinftyuniformly}
\min_{\gamma \in \Gamma'_{v,n}} \kappa(\gamma) \to \infty, \quad \text{as }n \to \infty.
\end{equation}
For every $\gamma \in \Gamma'_{v,n}$ let $L_{\gamma}$ be the event that $\gamma
\subset T_2$ and that $v$ has the minimal age in $\gamma$. Let $Y_{\gamma}:=1_{L_{\gamma}}$.
We will show that for all $n \in \N$,
\begin{equation}
\label{eq: wgammaYgamma}
\mathbf{E} \left[\sum_{\gamma \in \Gamma'_{v,n}}\kappa(\gamma)Y_{\gamma} \right]
\le M, \text{ for some  constant }M=M(v), \text{ independent of }n.
\end{equation}
Note that (\ref{eq: wgammaYgamma}) in conjunction with (\ref{eq: wtoinftyuniformly})
imply that $\lim_{n \to \infty}\mathbf{E} \left[\sum_{\gamma \in \Gamma'_{v,n}}Y_{\gamma}
\right] = 0$. In particular, a union bound over all $\gamma \in \Gamma'_{v,n}$
yields that  $$\Pr[I_{v,n}]=\Pr \left[\bigcup_{\gamma \in \Gamma'_{v,n}} L_{\gamma} \right]\le
\sum_{\gamma \in \Gamma'_{v,n}}\Pr[L_{\gamma}] =\mathbf{E}\left[
\sum_{\gamma \in \Gamma'_{v,n}}Y_{\gamma}\right]  \to 0, \text{ as }n \to
\infty,$$ 
from which the assertion of the theorem follows.

We now prove (\ref{eq: wtoinftyuniformly}). For all $1<m \in \N$ we define  $n(m)$ to be the minimal integer satisfying $(1+\frac{1}{md_{v}-1})^{n(m)-1} \ge m$. If $\gamma \in \Gamma'_{v,n(m)}$ and $\max \{d_u : u \in \gamma \} \le md_{v}$, then by the choice of $n(m)$, we get that $\kappa(\gamma) \ge  \prod_{u \in \gamma \setminus \{v \}} \frac{d_{u}} {d_{u}-1} \ge (1+\frac{1}{md_{v}-1})^{n(m)-1}  \ge m$. We now show that for every $\gamma \in \Gamma'_{v,n(m)}$ with $\max \{d_u : u \in \gamma \} > md_{v}$, we also have that $\kappa(\gamma)\ge m$. Fix some $\gamma \in \Gamma'_{v,n(m)}$. Consider \[J:=\{ 1<i \in [n+1] : d_{\gamma_{i}}>d_{\gamma_{j}} \text{ for all } j \in [i-1]\}.\] We can order the elements of $J$, as follows: $J=\{i_1,i_2,\ldots,i_k \}$  ($k=|J|$), such that if $1 \le s<t \le k$, then the degree of $v_{i_{s}}$ is smaller than that of $v_{i_{t}}$. For typographical reasons, for any $s \in [k]$ we denote $u_s:=\gamma_{i_s}$  and set $u_0=v$.

Note that if $1 \le x \le y \le z$, then $\frac{x+y-1}{x}
\frac{y+z-1}{y} \ge \frac{x+z-1}{x}$. Whence by induction, $\prod_{i=1}^{k}\frac{n_{i-1}+n_{i}-1}{n_{i-1}}
\ge \frac{n_0+n_{k}-1}{n_0}$, for any integers $1 \le n_0 < n_1
< \cdots < n_{k}$. Thus, $$\kappa(\gamma) \ge \prod_{s \in [k]}\frac{d_{\gamma_{i_{s}-1}}+d_{\gamma_{i_{s}}}-1}{d_{\gamma_{i_{s}-1}}} \ge  \prod_{s \in [k]}\frac{d_{u_{s-1}}+d_{u_s}-1}{d_{u_{s-1}}} \ge \frac{d_{u_{0}}+d_{u_{k}}-1}{d_{u_{0}}} \ge \frac{\max \{d_u : u \in \gamma \}}{d_v} \ge m,$$
as claimed. This establishes (\ref{eq: wtoinftyuniformly}).

We now prove (\ref{eq: wgammaYgamma}). Fix some $n \ge 2$ and some $\gamma\in \Gamma'_{v,n}$. Let $B_{\gamma}$ be the event that  $X_{\gamma_{i}}<X_{\gamma_{i+1}}$ for all $i \in [n]$ and that for all 
$1<i<n+1$ we have that  $X_{\gamma_{i}}<X_u$ for all $u \notin \gamma$ such that $u \sim \gamma_i$. By Lemma \ref{lem:monotone}, $B_{\gamma} \supset L_{\gamma}$. For every $1<i \le n$, we denote $$T_{i}(\gamma):=\left\{u : d\left(u,\{\gamma_i,\gamma_{i+1},\ldots,\gamma_n  \}\right) \le 1 \right\} \setminus \{\gamma_{i-1} \} .$$ Set $T_1(\gamma):=T_2(\gamma) \cup \{v \}$. For   $i \in [n]$ let $C_i(\gamma)$ be the event that $ \gamma_i$ is the youngest vertex in  $T_i(\gamma)$. Note that $B_{\gamma}=\bigcap_{i=1}^{n}C_{i}(\gamma)$. Observe that the events $C_{1}(\gamma),\ldots,C_n(\gamma)$  are mutually independent. One way to see this is to note that the conditional distribution of $(X_{u} : u \in T_{i+1}(\gamma))$, given $C_1(\gamma),\ldots,C_i(\gamma)$ and
$(X_{\gamma_{j}} : j \in [i])$, is that of independent Uniform$(X_{\gamma_{i}},1]$ random variables. Alternatively, this follows from the fact that all orderings of $T_{1}(\gamma)$ (with respect to the ages of the vertices of $T_{1}(\gamma)$) are equally likely. Thus, \begin{equation}
\label{eq: B(gamma)isopen}
\Pr[B_{\gamma}]= \prod_{i=1}^{n}\Pr[C_{i}(\gamma)]=\prod_{i=1}^{n}\frac{1}{|T_{i}(\gamma)|}. \end{equation}
Let $m(\gamma):=\{u \sim \gamma_{n+1} : d(u,\{\gamma_j : j \in [n] \}) \ge 2 \}$. For every $\gamma':=(\gamma'_1,\gamma'_2,\ldots,\gamma'_{n+2}) \in
\Gamma'_{v,n+1}$ we denote $\gamma'|_{[n+1]}:=(\gamma'_1,\ldots,\gamma'_{n+1}) $. From the definition of $\Gamma'_{v,n+1}$,  if $ \gamma'  \in \Gamma_{v,n+1}$ is such that $\gamma'|_{[n+1]}=\gamma$, then $\gamma'_{n+2} \in m(\gamma)$. Moreover, the following hold. $$\kappa(\gamma')=\kappa(\gamma)\frac{d_{\gamma_{n+1}}(d_{\gamma_{n}}+d_{\gamma_{n+1}}-1)} {(d_{\gamma_{n+1}}-1)d_{\gamma_{n}}},$$ $$|T_{i}(\gamma)| \le |T_{i}(\gamma')|, \text{ for all }i \in [n-1],$$ $$|T_{n+1}(\gamma')|=d_{\gamma_{n+1}},\, |T_n(\gamma)| = d_{\gamma_n} \text{ and }|T_n(\gamma')| \ge d_{\gamma_{n}}+|m(\gamma)|,$$ $$|m(\gamma)| \le d_{\gamma_{n+1}}-1.$$ Hence, by (\ref{eq: B(gamma)isopen})
we have\begin{equation}
\label{eq: B(gamma)isopen2}
\begin{split}
& \sum_{\gamma' \in \Gamma'_{v,n+1} : \gamma'|_{[n+1]}=\gamma} \Pr[B_{\gamma'}]\kappa(\gamma')= \sum_{\gamma' \in \Gamma'_{v,n+1} : \gamma'|_{[n+1]}=\gamma} \kappa(\gamma') \prod_{i=1}^{n+1}1/|T_{i}(\gamma')|
\\ & \le \kappa(\gamma)\frac{d_{\gamma_{n+1}}(d_{\gamma_{n}}+d_{\gamma_{n+1}}-1)} {(d_{\gamma_{n+1}}-1)d_{\gamma_{n}}} \prod_{i=1}^{n}|T_{i}(\gamma)|^{-1} \sum_{\gamma' \in \Gamma'_{v,n+1} : \gamma'|_{[n+1]}=\gamma}\frac{|T_n(\gamma)|}{d_{\gamma_{n+1}}|T_n(\gamma')|}   \\ & \le \kappa(\gamma)\frac{d_{\gamma_{n+1}}(d_{\gamma_{n}}+d_{\gamma_{n+1}}-1)} {(d_{\gamma_{n+1}}-1)d_{\gamma_{n}}}  \Pr[B_{\gamma}]|m(\gamma)| \frac{d_{\gamma_{n}}}{d_{\gamma_{n+1}}(d_{\gamma_{n}}+|m(\gamma)|)} \le  \Pr[B_{\gamma}]\kappa(\gamma).
\end{split}
\end{equation}
Denote $Z_{\gamma}:=1_{B_{\gamma}}$. From  (\ref{eq: B(gamma)isopen2}) we get the following recurrence relation,
\begin{equation}
\label{eq: sumwgammaYgamma}
\begin{split}
\mathbf{E} \left[\sum_{\gamma '\in \Gamma'_{v,n+1}}\kappa(\gamma')Z_{\gamma'} \right]
 \le \mathbf{E} \left[\sum_{\gamma \in \Gamma'_{v,n}}\sum_{\gamma '\in \Gamma'_{v,n+1} : \gamma'|_{[n+1]}=\gamma}\kappa(\gamma')Z_{\gamma'} \right] \le \mathbf{E} \left[\sum_{\gamma \in \Gamma'_{v,n}}\kappa(\gamma)Z_{\gamma} \right].
\end{split}
\end{equation}
Iterating, we get that $\mathbf{E} \left[\sum_{\gamma '\in \Gamma'_{v,n+1}}\kappa(\gamma')Z_{\gamma'} \right]\le \mathbf{E} \left[\sum_{\gamma
'\in \Gamma'_{v,2}}\kappa(\gamma')Z_{\gamma'} \right]  =: M(v)$, for every $1<n \in \N$. This implies (\ref{eq: wgammaYgamma}), since $Y_{\gamma'} \le Z_{\gamma'}$ for all $\gamma' \in
\Gamma_{v,n+1}$ (as $B_{\gamma'} \supset L_{\gamma'}$).
\end{proof}
\section{The first 4 layers in $\mathbb{Z}^d$}
\label{sec: Zd}

For every $1 \le i \le d$, let $e_{i} \in \mathbb{Z}^{d}$ be the vector whose
$i$th co-ordinate is 1 and the rest of its co-ordinates are 0. Let $$\Z_{+}^{d}:=\{(x_1,\ldots,x_d)
\in \Z^d : x_i \ge 0, i \in [d] \}.$$
In this section we prove the following theorem, whose assertion is stronger than that of Theorem \ref{thm: Z^dintro}. 
\begin{theorem}
\label{thm: Z^d}
For all sufficiently large $d$, the vertex-induced graph on $T_4(\Z^{d}) \cap \Z_{+}^{d}$ $\mathrm{a.s.}$ contains an infinite path $(v_1,v_2,\ldots)$, such that
$v_{i+1}-v_{i} \in \{ e_{j}: j \in [d]\}$, for all $i$. 
\end{theorem}
In the proof of Theorem \ref{thm: Z^d} we use a variant of the EIT method,
introduced in $\cite{cox1983oriented}$. Cox and Durrett attribute the argument
to Harry Kesten. As a warm-up, we first present in Lemma \ref{lem: EIT1}, a calculation
taken from  $\cite{cox1983oriented}$. We do not use Lemma \ref{lem: EIT1}
and we present it and its proof since the proof of Lemma \ref{lem: EIT2} uses some of the calculations from the proof of Lemma \ref{lem: EIT1}.

\begin{definition}[EIT]
\label{def: EITdef}
Let $\mu$ be a probability measure on infinite simple paths in a graph $G$. Let $\alpha \in (0,1)$. We say
that $\mu$ has $\mathrm{EIT}(\ga)$, if there exists some $C>0$ such that for all $k \in \N$, $$\mu \times
\mu (\{(\gamma,\gamma') : |\gamma \cap \gamma'| \ge k \}) \le C \ga^k,$$ where $|\gamma \cap \gamma'|$ is the number of vertices the
paths $\gamma$ and $\gamma'$ have in common. In simple words, the probability
that two paths chosen independently, each from the distribution $\mu$, will
have at least $k$ common vertices, is at most $C \ga^k$. If such $\mu$ exists,
we say that $G$ admits random paths with $\mathrm{EIT}(\ga)$. The same definition
applies when $G$ is an oriented graph and the paths are oriented paths. 
\end{definition}
\begin{definition}
\label{def: Monotonepaths}
 Consider a random walk with initial position $\mathbf{0}:=(0,0,\ldots,0)$, whose increments distribution is the uniform distribution on $\{e_{i}
:i \in [d] \}$. Let $\mu_{d}$ be the  probability measure corresponding to the infinite trajectory
of this random walk. Then $\mu_{d}$ is called the uniform distribution
on \emph{monotone paths} in $\mathbb{Z}^{d}$. We call a path $\gamma$ a \emph{monotone path} if $\gamma_{i+1}-\gamma_{i}\in \{e_j:j \in [d] \} $, for all $i\in \N$. We denote the collection of all monotone paths of length $k$ starting from $\mathbf{0}$ by $\Gamma_k^{\mathrm{mon}}$.\end{definition}

\begin{lemma}
\label{lem: EIT1}
$\mu_{d}$ has $\mathrm{EIT}(\ga_d)$ for some $\ga_d=1/d+(1/d)^{2}+O(d^{-3})$, for
any $d \ge 4$.
\end{lemma}

\begin{proof}
Let $(S_k)_{k=0}^{\infty}$ and $(S'_k)_{k=0}^{\infty}$ be two independent random walks with distribution
$\mu_d$.
Let $\tau:=\inf \{k \ge 1 : S_{k}=S'_k \}$. Then,
\begin{equation}
\label{eq: 4.1}
\begin{split}
\Pr(\tau=1)&=d^{-1},
\\ \Pr(\tau=2)&=d^{-2}-d^{-3},
\\ \Pr(\tau=3)& <
3d^{-3},
\end{split}
\end{equation}
as
\begin{equation*}
\begin{split}
& \Pr(\tau=3) = \\ & \Pr[ d(S_1,S'_1)=2] \Pr[ d(S_2,S'_2)=2 \mid d(S_1,S'_1)=2]\Pr[ d(S_3,S'_3)=0 \mid d(S_2,S'_2)=2] < 3d^{-3}.
\end{split}
\end{equation*}
By the independence of  $(S_k)_{k=0}^{\infty}$ and $(S'_k)_{k=0}^{\infty}$,  \begin{equation}
\label{eq: 4.4'}
\Pr(\tau=k) \le \Pr(S_k=S'_k) = \sum_{x}\Pr(S_{k}=x)\Pr(S'_k=x) \le \max_x
\Pr(S_{k}=x), \text{ for any }k \ge 4.
\end{equation}
If $4 \le k \le d$, then for any $x \in \Z_+^d$, with $d(x,0)=k$, we have
\begin{equation}
\label{eq: 4.2'}
\Pr(S_k=x)={k \choose x_{1},\ldots,x_{d}}d^{-k} \le k!d^{-k}.
\end{equation}
Since the right hand side of (\ref{eq: 4.2'}) is non-increasing in $k$ for $4 \le k \le d$,
  (\ref{eq: 4.4'})-(\ref{eq: 4.2'}) imply that,
\begin{equation}
\label{eq: 4.3'}
\Pr(4 \le \tau \le d) \le \sum_{k=4}^{d}k!d^{-k} \le 4!d^{-4}+5!d^{-5}+(d-6)6!d^{-6}=O(d^{-4}).
\end{equation}

If $d\ell \le k < d(\ell+1)$ for some $\ell \in \N$, then for any  $x \in \Z_+^d$
with $d(x,0)=k$,
\begin{equation}
\label{eq: 4.5'}
P(S_k=x)={k \choose x_{1},\ldots,x_{d}}d^{-k} \le \frac{k!}{\ell!^{d-k+d\ell}(\ell+1)!^{k-d\ell}}d^{-k}=\frac{(d\ell)!}{\ell!^dd^{d\ell}}\prod_{i=1}^{k-d\ell}\frac{d\ell+i}{d(\ell+1)}\le
\frac{(d\ell)!}{\ell!^dd^{d\ell}}.
\end{equation} 
By Stirling's formula (e.g$.$ $\cite{feller2008introduction}$ page 54),
$$e^{-1/13} \le \frac{n!}{n^ne^{-n}\sqrt{2\pi n}} \le 1, \text{ for all }n \in \N.$$
Plugging this estimate in (\ref{eq: 4.5'}), we get by (\ref{eq: 4.4'}) that:
\begin{equation}
\label{eq: 4.6'}
\begin{split}
\Pr(\tau > d) \le \sum_{\ell=1}^{\infty} \Pr(d\ell \le \tau < d(\ell+1)) \le \sum_{\ell=1}^{\infty}d\frac{(d\ell)!}{\ell!^dd^{d\ell}}
\\ \le \sum_{\ell=1}^{\infty} \frac{(dl)^{d\ell}d(2 \pi d\ell)^{1/2}}{(\ell^{\ell}e^{-\frac{1}{13}}\sqrt{2
\pi \ell})^{d}d^{d\ell}} = d^{3/2}\sqrt{2 \pi }\left(\frac{e^{\frac{1}{13}}}{\sqrt{2
\pi}}\right)^{d} \sum_{\ell=1}^{\infty}\ell^{(1-d)/2}.
\end{split}
\end{equation} 
The last sum is finite since $d \ge 4$. Since $e^{1/13} < \sqrt{2 \pi}$,
the last expression approaches zero exponentially rapidly as $d \to \infty$.

In conclusion, $\Pr(\tau < \infty) \le 1/d+(1/d)^{2}+O(d^{-3})$.
The assertion of the lemma now follows from the strong Markov property, applied to the random walk  $(S_k-S'_k)_{k=0}^{\infty}$.
\end{proof}
We adopt the convention that for any $i \in \N$, $O(d^{-i})$ can be a negative term whose absolute value is $O(d^{-i})$ in the usual sense.
\begin{lemma}
\label{lem: EIT2}
Let  $(S_k)_{k=0}^{\infty}$ and $(S'_k)_{k=0}^{\infty}$ be two independent random walks with distribution
$\mu_d$. Denote
\begin{equation*}
\begin{split}
p_{1,2,3,4}&:=\Pr[\{\exists n>2,
d(S_n,S'_n)=2 \} \mid S_2=e_1+e_{3},S'_2=e_2+e_{4}], \\ p_{1,2}&:=\Pr[\{\exists
n>2, d(S_n,S'_n)=2 \} \mid S_2=2e_1,S'_2=2e_2], \\ p_{1,2,3}&:=\Pr[\{\exists n>2,
d(S_n,S'_n)=2 \} \mid S_2=2e_1,S'_2:=e_2+e_{3}],
\\ a_2&:= \Pr[\{\exists n>1,
d(S_n,S'_n)=0 \} \mid  S_1=e_1,S'_1=e_2].
\end{split}
\end{equation*}
Then,
\begin{equation}
\label{eq: transitionprobabilities}
\begin{split}
a_{2} & = d^{-2}+O(d^{-3}),
\\p_{1,2}& = d^{-2}+O(d^{-3}),
\\ p_{1,2,3}& = 2d^{-2}+O(d^{-3}),
\\ p_{1,2,3,4}& = 4d^{-2}+O(d^{-3}).
\end{split}
\end{equation}
\end{lemma}
\begin{proof}
Let $\tau:=\inf \{k \ge 1 : S_{k}=S'_k \}$. By symmetry of the lattice $$a_2=\Pr[\tau<\infty \mid \tau>1]=\frac{\Pr[1<\tau<\infty
]}{\Pr[\tau>1]} = \frac{d^{-2}+O(d^{-3})}{1-d^{-1}}=d^{-2}+O(d^{-3}).$$ Let $\tau':=\tau-4.$
Note that if $d(S_2,S'_2)=4$, then $\tau' \ge 0$.  Moreover, $\Pr[\tau'=0 \mid  (S_2,S'_{2})=(2e_1,2e_2)]=d^{-4}$, $\Pr[\tau'=0 \mid  (S_2,S'_2)=(2e_1,e_2+e_{3})]=2d^{-4}$ and $\Pr[\tau'=0 \mid  (S_2,S'_2)=( e_1+e_{3},e_2+e_{4})]=4d^{-4}$. Similarly to the proof of Lemma \ref{lem: EIT1}, for any $(z,z') \in \{(2e_1,2e_2),(2e_1,e_2+e_{3}),( e_1+e_{3},e_2+e_{4}) \}$ we have that
\begin{equation*}
\begin{split}
& \Pr[1 \le \tau' < \infty \mid (S_2,S'_{2})=(z,z') ] \le \sum_{k=3}^{\infty}\Pr[ S_{k+2}-S'_{k+2}=\mathbf{0}
\mid (S_2,S'_{2})=(z,z') ]=  \\ & \sum_{k=3}^{\infty}\Pr[ S_{k}-S'_{k}=z'-z] \le 3!^{2}d^{-5}+4!^{2}d^{-6} +
\sum_{k=5}^{\infty}\max_{x} \Pr(S_{k}=x)=O(d^{-5}).
\end{split}
\end{equation*}
 By the strong Markov property, $$p_{1,2}:=\frac{\Pr[\tau'<\infty \mid  (S_2,S'_{2})=(2e_1,2e_2)]}{a_{2}}=\frac{d^{-4}+O(d^{-5})}{d^{-2}+O(d^{-3})}=d^{-2}+O(d^{-3}).$$
The proofs of the last two equations in (\ref{eq: transitionprobabilities}) are concluded in the same manner.  
\end{proof}

\emph{Proof of Theorem \ref{thm: Z^d}:}
Let $\gamma \in \Gamma_{2k-1}^{\mathrm{mon}}$ (where $\Gamma_{2k-1}^{\mathrm{mon}}$ is as in Definition \ref{def: Monotonepaths}). To have symmetry in our construction, for reasons that shall soon be clear, we define $\gamma_{0}:=-e_1$ and $\gamma_{2k+1}:=\gamma_{2k}+e_1$. For any $i \in [k]$, we define $$M_{2i-1}(\gamma):=|\{u \sim \gamma_{2i-1} : u \neq \gamma_{2i-2}, X_{u}<X_{\gamma_{2i-1}}
\}|,$$ and $$M_{2i}(\gamma):=|\{u
\sim \gamma_{2i} : u \neq \gamma_{2i+1}, X_u<X_{\gamma_{2i}} \}|.$$ Define $$A_{i}(\gamma):=\{M_{2i-1} (\gamma)\le 2 \} \cap \{M_{2i} (\gamma)\le 2 \},$$ and denote $$A(\gamma):=\bigcap_{i=1}^{k}A_i(\gamma).$$ By construction, on $A(\gamma)$, $\gamma$ is contained in $T_4(\Z^d)$. Notice that  $\gamma_{2i-1}$ and $\gamma_{2i}$ do not have a common neighbor, for all $i \in [k]$. Whence, similarly to the proof of Theorem \ref{thm: T3fortrees}, we can calculate $\Pr[A_{i}(\gamma)]$ by a direct calculation which is completely elementary.
\begin{proposition}
\label{prop: marginalsZd}
For all $i \in [k]$ we have that $\Pr[A_{i}(\gamma)] \ge (a /d)^{2}$ for some absolute
constant $a>1$.
\end{proposition}
\emph{Proof of Proposition \ref{prop: marginalsZd}:} For the sake of concreteness, we show that for $ i \in [k]$, 
\begin{equation*}
\Pr[A_{i}(\gamma)]=\left(\frac{2}{2d-1} \right)^{2}+\frac{(2d-3)}{(4d-2)(4d-3)(2d-1)}+\frac{3(2d-3)}{(4d-2)(4d-3)(4d-5)} >\frac{9}{8d^{2}}.
\end{equation*}
 Before explaining this inequality we note that the exact term $9/8$ shall not be used in what comes. Denote
\begin{equation*}
M_{j}'(\gamma):=|\{u \sim \gamma_{j} : u \notin \gamma, X_{u}<X_{\gamma_{j}}\}|,\, j=2i-1,2i.
\end{equation*}
The first term above comes from the case that $M_{2i-1}'(\gamma),M_{2i}'(\gamma) \le 1$. The middle term comes from the case that $M_{2i}' (\gamma)= 2$, $M_{2i-1}' (\gamma)\le 1$ and $\gamma_{2i}$ is younger than all of the vertices in $\{u: u \sim \gamma_{2i-1}, u \notin \gamma \} \cup \{\gamma_{2i-1} \}$ and the corresponding case in which the roles of $2i$ and $2i-1$ are reversed. The last term comes from the case that $M_{2i}' (\gamma)= 2$, $M_{2i-1}' (\gamma)= 1$ and the unique neighbor of $\gamma_{2i-1}$ not belonging to $\gamma$ which is younger than $\gamma_{2i-1}$  is also younger than $\gamma_{2i}$ which in turn is younger than $\gamma_{2i-1}$, together with the corresponding case where the roles of $2i$ and $2i-1$ are reversed. \qed
\vspace{2mm}

Note that the events $A_1(\gamma),\ldots,A_k(\gamma)$ are usually not independent, since vertices of distance two in $\gamma$ can have a common neighbor. Hence we need the following proposition. Recall that in our convention $v \sim A \subset V$ iff $d(v,A)=1$ (in particular $v \notin A$).
\begin{proposition}
\label{prop: pathsproband domination}
Let $\gamma \in \Gamma_{2k-1}^{\mathrm{mon}}$. Denote $X:=(X_a : a \sim \gamma)$, $Y:=(X_u : u \in \gamma)$.
\begin{itemize}
\item[(i)]
The events $A_1(\gamma),\ldots,A_k(\gamma)$ are positively correlated.
\item[(ii)]
The conditional distribution of $X$, given $A(\gamma)$, stochastically dominates its unconditional distribution.
\end{itemize} 
\end{proposition}
\emph{Proof of Proposition \ref{prop: pathsproband domination}:}  Fix some $\gamma :=(v_1,\ldots,v_{2k}) \in \Gamma_{2k-1}^{\mathrm{mon}}$. Let $[0,1]^{\gamma}$ (resp.~$[0,1]^{\{a:a \sim \gamma \}}$) be the collection of vectors whose co-ordinates
take values in $[0,1]$ and
are labeled by the set $\gamma$ (resp.~$\{a:a
\sim \gamma \}$). For any $j \in [k]$, let
$f_{j}(X,Y)$ be the indicator of $A_j(\gamma)$. Observe that $f_{1}(X',Y'),\ldots,f_{k}(X',Y')$ are increasing
functions of $X'$ ($X' \in [0,1]^{\{a:a
\sim \gamma \}} $), for any fixed
$Y' \in [0,1]^{ \gamma }$.

For any $w\in\{1,2\}^k$, we define a partial order $\prec_{w} $ on $\gamma$ as follows. For any $j \in [k]$, if $w(j)=1$, then $v_{2j} \prec_{w} v_{2j-1}$ and if    $w(j)=2$, then $v_{2j-1} \prec_{w} v_{2j}$ (and these are the only relations in $\prec_{w}$). Let
$S_{w}$ be the event that for all $i \in [k]$, $X_{v_{2i-1}} < X_{v_{2i}}$ iff $v_{2i-1} \prec_{w} v_{2i}$. For any $w \in \{1,2\}^k$ let $Y_w$ be a random vector
distributed as $Y$ conditioned
on $S_w$. Denote by $\mathbf{E}_{X}$ the expectation with respect
to $X$ where
$Y$ (or $Y_{w}$) is treated as a constant
vector. For each  $w\in\{1,2\}^k$, we say that $Z \in [0,1]^{\gamma}$ respects $\prec_w$ if  $Z_{v_{2i-1}} < Z_{v_{2i}}$
iff $v_{2i-1} \prec_{w} v_{2i}$.

  Let $Z_1,Z_2 \in [0,1]^{\gamma}$ and $w\in\{1,2\}^k$. Note that
if $Z_1 \ge Z_2$ coordinate-wise and both vectors respect $\prec_w$, then for any $j \in [k]$, $f_j(X',Z_{2}) \ge f_j(X',Z_{1})$,
for all $X' \in [0,1]^{\{a:a \sim \gamma \}}$. Consequently, $\mathbf{E}_{X}[f_j(X,Y_{w})]$  (a shorthand for $\mathbf{E}_{X}[f_j(X,Y) \mid S_w]$) is a decreasing function of $Y_{w}$, for any
$j \in [k]$ and $w\in \{1,2\}^k$. Let $\mathbf{E}_{Y_{w}}$ be the expectation
with respect to $Y$ conditioned on $S_w$ (that is, $\mathbf{E}_{Y_w}[\cdot]=\mathbf{E}_{Y}[\cdot \mid S_w] $). Observe that by symmetry we have that  $$\Pr[A_{i}(\gamma) \mid S_w]=\Pr[A_{i}(\gamma) ] , \text{ for any }i \in [k] \text{ and } w \in \{1,2\}^k.$$ Let $G(X)$ be the indicator function of some  increasing event
$B$ with respect to $X$. By an application
of the FKG inequality (first and second inequalities) and of the correlation inequality for \emph{affiliated} random variables from Theorem \ref{thm: affiliation} (third inequality) we have that,
\begin{equation}
\label{eq: FKGonpaths}
\begin{split}
& \Pr \left[B \cap  A(\gamma) \cap S_w \right] =\Pr[S_w] \mathbf{E}_{Y_w,X} \left[g(X) \prod_{j
\in [k]} f_j (X,Y_w) \right] \\ & =\Pr[S_w] \mathbf{E}_{Y_{w}}\mathbf{E}_{X} \left[g(X) \prod_{j \in [k]} f_j (X,Y_w) \right] \ge \Pr[S_w]\mathbf{E}[g(X)] \mathbf{E}_{Y_{w}}\mathbf{E}_{X} \left[ \prod_{j \in [k]} f_j (X,Y_{w})\right] \\ & \ge \mathbf{E}[g(X)]\Pr[S_w]\mathbf{E}_{Y_{w}} \left[ \prod_{j \in [k]} \mathbf{E}_{X}[f_j (X,Y_{w})] \right] \ge \Pr[B]\Pr[S_w] \prod_{j \in [k]} \mathbf{E}_{Y_{w}} \mathbf{E}_{X}[f_j (X,Y_{w})]   \\ & =  \Pr[B] \Pr[S_w] \prod_{j \in [k]}\Pr[A_{j}(\gamma) \mid S_{w}] = \Pr[B] \Pr[S_w] \prod_{j \in [k]}\Pr[A_{j}(\gamma)].
\end{split}
\end{equation}
Taking $B$ to equal the entire probability space and  summing over all $w \in \{1,2\}^k$  give that
\begin{equation}
\label{eq: poscor}
\Pr[A(\gamma)]=\Pr \left[\bigcap_{j \in [k]}A_{j}(\gamma) \right] \ge  \prod_{j \in [k]}\Pr[A_{j}(\gamma)] .
\end{equation}
Similarly, for any disjoint $I_1,I_2 \subset [k]$, by repeating the calculations in (\ref{eq: FKGonpaths}) and the reasoning leading to (\ref{eq: poscor}), one can show that if we denote $J_i:=\bigcap_{j \in I_i}A_{j}(\gamma)$ ($i=1,2$), then
\begin{equation}
\label{eq: poscor2}
\Pr[J_{1} \cap J_2] \ge \Pr[J_{1}]\Pr[J_2].  
\end{equation}

Moreover, by the first inequality in (\ref{eq: FKGonpaths}),
\begin{equation*}
\Pr \left[B \cap  A(\gamma) \cap S_w \right] \ge  \Pr[S_w]\mathbf{E}[g(X)]
\mathbf{E}_{Y_{w}}\mathbf{E}_{X} \left[ \prod_{j \in [k]} f_j (X,Y_{w})\right] =\Pr[B] \Pr \left[A(\gamma)\cap S_{w}\right],
\end{equation*}
for any $w \in   \{1,2\}^{[k]}$. Summing over all  $w$ we get that $\Pr \left[B \mid A(\gamma)  \right] \ge \Pr[B]$. \qed 
\vspace{2mm}

Set
\begin{equation}
\label{defofZk}
Y_{ \gamma}:=\frac{1_{A(\bar \gamma)}}{\Pr[A(\bar \gamma)]} \text{ for all }\gamma \in \Gamma_{2k-1}^{\mathrm{mon}} \text{ and } Z_{k}:=|\Gamma_{2k-1}^{\mathrm{mon}}|^{-1} \sum_{\bar \gamma \in
\Gamma_{2k-1}}Y_{\bar \gamma}.
\end{equation}
Clearly $\mathbf{E}[Z_{k}]=1$. Similarly to the proof of Theorem \ref{thm: T3fortrees}, in order to prove the assertion of Theorem \ref{thm: Z^d} it suffices to show that for some positive constant $\beta$ (which may depend only on $d$) we have that
\begin{equation}
\mathbf{E}[Z_k^2] \le \beta,
\end{equation}
since then, by Cauchy-Schwarz inequality (or Theorem \ref{thm: payleyzygmund}
for $r=0$)
\begin{equation}
\label{eq: CS} \Pr[Z_{k}>0] \ge \frac{1}{\mathbf{E}[Z_k^2]} \ge \beta^{-1}>0.
\end{equation}
The event that $T_4(\Z^{d}) \cap \Z_{+}^{d}$ contains an infinite monotone path is a tail event. Since it contains the decreasing intersection of the events $(\{Z_k>0\}:k\in \N)$, 
(\ref{eq: CS}) and
the 0-1 law of Lemma \ref{lem: 0-1infinitecluster} imply the assertion of the theorem.

Let $\gamma=(v_1,\ldots,v_{2k}), \gamma'=(v'_1,\ldots,v'_{2k}) \in \Gamma_{2k-1}^{\mathrm{mon}}$. As in the proof of Theorem \ref{thm: T3fortrees}, in order to estimate $\mathbf{E}[Z_k^2] $ from above, we need to estimate $\Pr[A(\gamma)\cap A(\gamma')]/(\Pr[A(\gamma)]\Pr[A(\gamma')]) $ from above.

As before, set $v_{0},v'_0:=-e_1$ and $v'_{2k+1}:=v'_{2k}+e_1$, $v_{2k+1}:=v_{2k}+e_1$. Observe that if $i \in [2k]$ and $d(v_i,v'_i)=2$, then given $A(\gamma)$,
the conditional distribution of $X_{v'_{i}}$ is different than its unconditional distribution only if $d(v'_i,\{v_{i-1},v_{i+1}\})=1$, in which case by Proposition \ref{prop: pathsproband domination} (ii), its conditional distribution stochastically dominates   its unconditional distribution. This can only decrease the probability of the event $A_{\lceil i/2 \rceil}(\gamma')$ (hence plays in our favor).

For any $i \in [2k]$, set
\begin{equation}
\label{eq: N}
N_{i}(\gamma,\gamma')=N_{i}:=\{u \sim v'_i : u \notin \gamma'\cup \{v'_{2k+1} \} , d(u,\gamma) \le 1 \}.
\end{equation}
If $d(v_i,v'_i)=2$, then $|N_i| \le 4$. Note that if $d(v_i,v'_i)=2$, then the distribution of $X_{v'_i}$ and of $(X_u:u \sim v'_i, u \notin N_i \cup \gamma )$ is
unaffected by the occurrence of the event $A(\gamma)$. Similarly to the proof of Theorem \ref{thm: T3fortrees} where we considered the event $B_{\gamma,\gamma'}$, we now define an event which contains $A(\gamma')$, whose conditional distribution, given $A(\gamma)$, is easier to estimate than that of $A(\gamma')$. We define $M_j(\gamma')$ in an analogous manner to the definition of  $M_j(\gamma)$. Namely,  $M_{2i-1}(\gamma'):=|\{u
\sim v'_{2i-1} : u \neq v'_{2i-2}, X_{u}<X_{v'_{2i-1}}
\}|$ and $M_{2i}(\gamma'):=|\{u
\sim v'_{2i} : u \neq v'_{2i+1}, X_u<X_{v'_{2i}} \}|$.  For any $j \in [2k]$, we define
\begin{equation}
\label{eq: K_j}
K_j:=\begin{cases}M_{j}(\gamma), & d(v_j,v_j')=0, \\
|\{u \sim v'_j : u \notin \gamma
\cup N_j,u \nsim \gamma' \setminus \{v_j' \}, X_u<X_{v'_j} \}|, & d(v_j,v_j')=2, \\
M_j(\gamma'), & d(v_j,v_j')>2. \\
\end{cases}
\end{equation}
Define $C_j=C_{j}(\gamma,\gamma'):=\{K_j \le 2 \}$ and $C=C(\gamma,\gamma'):=\bigcap_{i \in [2k]}C_j$. Note that $ A(\gamma) \cap C\supset A(\gamma) \cap A(\gamma')  $. Using the above observations, we get the following inequality.
\begin{proposition}
\label{prop: PA(gamma)intersectionA(gamma')}
Assume $d>4$. Let $a$ be as in Proposition \ref{prop: marginalsZd}.  Let $j_0,j_2 \in \N$ be such that $j_0+j_2 \le 2k$. Assume that $v_i=v'_i$, for any $1 \le i \le j_0$, that $d(v_i,v'_i)=2$, for any $j_0+1 \le i \le j_0 + j_2$, and that $d(v_i,v'_i) \ge 4$, for any $j_{0}+j_{2}+1 \le i \le 2k$.  Then there exists an absolute constant $a'$ such that $1<a' \le a$ and
\begin{equation}
\label{eq: pathsintersectionprobability}
\Pr \left[ A(\gamma)\cap C(\gamma,\gamma') \right]  \le \Pr \left[ A(\gamma) \right]\Pr \left[ A(\gamma') \right](d/a')^{j_{0}}\left(\frac{3d}{a'(2d-7)}
\right)^{j_{2}}.
\end{equation}
\end{proposition}
We note that the exact value of $a'$ is not important for our application, and the key point is that $d/a'<d$ and that for $d
\ge 8$, we have that $\frac{3d}{a'(2d-7)} <d/3< p^{-1}  $, where in the notation
of Lemmata \ref{lem: EIT1} and \ref{lem: EIT2}, $p:=\Pr[d(S_{i+1},S'_{i+1})=2
\mid d(S_{i},S'_{i})=2]$. 
\vspace{2mm}

\emph{Proof of Proposition \ref{prop: PA(gamma)intersectionA(gamma')}:} For simplicity, assume that both $j_0$ and $j_2$ are divisible
by 2, in which case we can take $a'=a$. The other cases are treated in an analogous manner and the possible small difference in the probabilities between the cases can be absorbed by taking some sufficiently smaller $a'$ such that $1<a'<a$.  

We first observe that by construction (namely, by (\ref{eq: K_j})) if $j_0+j_2 +1\le 2j-1$, then $C_{2j-1}\cap C_{2j}=A_j(\gamma')$. Denote,
\begin{equation*}
D_0:=\bigcap_{j=1}^{j_0}C_j,\, D_2:=\bigcap_{j=j_{0}+1}^{j_0+j_2}C_j \text{ and } D_{\ge 4}:=\bigcap_{j=j_{0}+j_2+1}^{2k-1}C_j.
\end{equation*}
We argue that the events $A(\gamma) \cap D_{2}$ and  $D_{\ge 4}$ are independent. To see this, first note that $d(\{v'_{i} : j_0+j_2+1 \le i \le 2k-1 \},\gamma) \ge 3$. From the definition of $K_j$ in the case that $d(v_j,v'_j)=2$, the event $D_{2}$ does not depend on the ages of the vertices in $\{v'_{i} : j_0+j_2+1 \le i \le 2k-1 \} $ and also does not depend on the ages of any of their neighbors, apart from $v'_{j_0+j_2}$, whose age is irrelevant for $D_{\ge 4}$.

Moreover, from the definition of $K_i$ when $d(v_i,v_i')=0$, $D_0 \subset A(\gamma)$.
Hence
\begin{equation}
\label{eq: D2A(gamma)ind}
\Pr[A(\gamma) \cap C]=\Pr[A(\gamma) \cap D_2 \cap D_{\ge 4}]=\Pr[A(\gamma)]\Pr[D_2 \mid A(\gamma) ]\Pr[D_{\ge 4}]
\end{equation}
Let $(X_{v'_{j}}':j_0+1 \le j
\le j_0+j_2)$ be i.i.d$.$ $\mathrm{Uniform}[0,1]$ random variables, independent of $(X_v :v \in \Z^d)$. For any $j_0+1 \le j
\le j_0+j_2$, denote $$F_{j}(\gamma,\gamma')=F_j:=\{u
\sim v'_j : u \notin \gamma'
\cup N_j,u \nsim (\gamma' \setminus \{v_j' \}) \}$$ (where the set $N_j$ is defined in (\ref{eq: N})) and set $$K_j':=|\{u \in F_j : X_u<X_{v'_j}' \}|.$$ Define
\begin{equation*}
B_j:=\{K_j' \le 2 \} \text{ and set } B:=\bigcap_{j=j_0+1}^{j_0+j_2}B_j.
\end{equation*}
Note that by construction $|F_j| \ge 2d - 8$, for any $j_0+1 \le j
\le j_0+j_2$. Moreover, by
construction, the sets $F_{j_{0}+1},\ldots,F_{j_{0}+j_2}$ are disjoint and are also disjoint from the
set of vertices that the event $A(\gamma) $ depends
on their ages. Using this and Proposition \ref{prop: pathsproband domination} (ii) we get that
\begin{equation}
\label{eq: D2}
\Pr \left[D_{2} \mid A(\gamma)\right] \le \Pr \left[B \mid A(\gamma)\right] \le \left(\frac{3}{2d-7}\right)^{j_2}.
\end{equation}
Denote
\begin{equation*}
A_0:=\bigcap_{j \in [j_{0}/2]}A_{j}(\gamma'),\,
A_2:=\bigcap_{j \in [\frac{j_{0}}{2}+1]}^{\frac{j_0+j_2}{2}}A_{j}(\gamma')
\text{ and } A_{\ge 4}:=\bigcap_{j \in [\frac{j_{0}}{2}+1]}^{\frac{j_0+j_2}{2}}A_{j}(\gamma').
\end{equation*}
Note that $A_{\ge 4}=D_{\ge 4}$. By Proposition \ref{prop: marginalsZd},
$\Pr[A_{0}] \ge (a/d)^{j_{0}}$, $\Pr \left[A_{2}
\right] \ge (a/d)^{j_{2}}$. By (\ref{eq: poscor2}) the events $A_0$, $A_2$ and $A_{\ge 4}$ are positively correlated. Thus  by (\ref{eq: D2A(gamma)ind}) and (\ref{eq: D2}) we get that
\begin{equation*}
\begin{split}
\Pr \left[A(\gamma) \cap C \right] & \le \Pr \left[A(\gamma) \right]\left(\frac{3}{2d-7}\right)^{j_2} \Pr \left[D_{\ge 4} \right] \\ & \le \Pr \left[A(\gamma)
\right] \Pr \left[A_{0} \right](d/a)^{j_{0}} \Pr \left[A_{2} \right](a^{-1}d)^{j_{2}} \left(\frac{3}{2d-7}\right)^{j_2} \Pr \left[A_{\ge 4} \right]  \\ & \le \Pr \left[ A(\gamma) \right]\Pr \left[ A(\gamma') \right](d/a)^{j_{0}}\left(\frac{3d}{a(2d-7)}
\right)^{j_{2}}. \qed 
\end{split}
\end{equation*}
Denote the uniform distribution
on $\Gamma_{2k-1}^{\mathrm{mon}}$ by $\nu_{2k-1}$.  Denote the expectation operator
with respect to $\nu_{2k-1} \times \nu_{2k-1}$  by $\mathbf{E}_{2k-1} \times \mathbf{E}_{2k-1}$. Pick two paths $\gamma=(v_1,\ldots,v_{2k}) $ and  $\gamma'=(v'_1,\ldots,v'_{2k})$
in $\Gamma_{2k-1}^{\mathrm{mon}}$ according to $\nu_{2k-1} \times \nu_{2k-1}$, where the choice of the paths is done
independently of the layers model, that is, independently of the ages
of the
vertices. Recall the definition of $Z_k$ in (\ref{defofZk}).
Observe that
\begin{equation}
\label{eq: 2ndmomentofZkviaEIT}
\begin{split}
\mathbf{E}[Z_k^2]&=|\Gamma_{2k-1}^{\mathrm{mon}}|^{-2} \sum_{\gamma,\gamma' \in \Gamma_{2k-1}}\frac{\Pr[A(\gamma)\cap A(\gamma')]}{\Pr[A(\gamma)]\Pr[A(\gamma')]}=\mathbf{E}_{2k-1} \times
\mathbf{E}_{2k-1}\left[ \frac{\Pr[A(\gamma)\cap
A(\gamma')]}{\Pr[A(\gamma)]\Pr[A(\gamma')]}\right] \\ & \le \mathbf{E}_{2k-1} \times
\mathbf{E}_{2k-1}\left[ \frac{\Pr[A(\gamma)\cap
C(\gamma,\gamma')]}{\Pr[A(\gamma)]\Pr[A(\gamma')]}\right]=:\rho_{k},
\end{split}
\end{equation}
where the probability inside the expectation is taken with respect to the layers model for fixed $\gamma,\gamma' \in \Gamma_{2k-1}^{\mathrm{mon}}$
 and the expectation
outside indicates that we take an average according to random $\gamma,\gamma'$ picked according to $\nu_{2k-1} \times \nu_{2k-1}$, independently of the layers model. By (\ref{eq: CS}) and (\ref{eq: 2ndmomentofZkviaEIT}), we only need to find a constant $\beta>0$ such that for any $k$, $w_k \le \beta$. We find such $\beta$ by combining Proposition \ref{prop: PA(gamma)intersectionA(gamma')} with Lemma \ref{lem: EIT2}.

Let $\gamma=(v_{1},\ldots,v_{2k})$ and $\gamma'=(v'_{1},\ldots,v'_{2k})$ be random paths in $\Gamma_{2k-1}^{\mathrm{mon}}$ chosen according to $\nu_{2k-1} \times \nu_{2k-1}$. We say that $\gamma$ and $\gamma'$ are at distance $r$ at time $i$ if $d(v_i,v'_i)=r$. We think about the paths as being exposed one vertex at a time according to the random walks from Definition \ref{def: Monotonepaths}.
Notice that for any $i \in [2k]$, $d(v_i,\gamma') \ge d(v_i,v'_i)-1 $. Whenever $v_{i-1} \neq v'_{i-1}
$ but $v_{i}=v'_{i}$, $\gamma$ and $\gamma'$ intersect each other for a random number of
times (including time $i$) which is stochastically dominated by a $\mathrm{Geometric}(1-d^{-1})$ random variable.

Whenever the two
paths are at distance 2 from each other, they have a chance of $\frac{3d-4}{d^2}$
to stay at distance 2 in the next step. Thus, similarly to the previous case,
whenever  $d(v_{i-1}, v'_{i-1}) \neq 2$, but $d(v_{i},v'_{i})=2$, the two
paths stay at distance 2 from each other for a random number of steps which
is stochastically dominated by a Geometric$(1-3d^{-1})$ random variable. If $d(v_i,v'_{i})=2$, then the conditional probability that $v_{i+1}=v'_{i+1}$,
given that $d(v_{i+1},v'_{i+1})\neq 2$ is by (\ref{eq: transitionprobabilities}) (and the symmetry of the lattice) $$q_{2,0}:=\frac{d^{-2}}{1-(3d-4)d^{-2}}=d^{-2}+O(d^{-3}).$$ Similarly, if   $d(v_i,v'_{i})=2$ then  the conditional probability that $d(v_{i+1},v'_{i+1})=4$, given that  $d(v_{i+1},v'_{i+1})\neq 2$, is  $$q_{2,4}:=1-d^{-2}-O(d^{-3}).$$ Once the paths reach distance 4 from each other, by (\ref{eq: transitionprobabilities})
(and the symmetry of the lattice) the probability
that the paths will ever be at distance $2$ again is  at most $$q_{4,2}:=4d^{-2}+O(d^{-3}).$$ So the number
of times the two paths will ever return from distance 4 to distance $2$ is
stochastically dominated by a $\mathrm{Geometric}(1-4d^{-2}-O(d^{-3}) )$ random variable.
As long as the paths are at distance at least $4$, the events $A_{i}(\gamma)$
and $C_{2i-i}(\gamma,\gamma') \cap C_{2i}(\gamma,\gamma')= A_{i}(\gamma')$ are independent. Hence, the only contributions to $\rho_k$ 
come from time intervals in which the paths stay at distance $0$ from
each other and from time intervals in which the paths stay at distance
2 from each other. To be more precise, one can extend Proposition \ref{prop: pathsproband domination} to cover also cases in which there may be several time intervals in which the paths are at distance 0 or 2, where now $j_0$ and $j_2$ would be the total number of steps that the paths spend at distance $0$ and $2$, respectively, from each other, so that (\ref{eq: pathsintersectionprobability}) still holds. We omit the details.

Let $s_0=s_0(\gamma,\gamma')$
and $s_2=s_2(\gamma,\gamma')$ be the number of time intervals in which the
paths are at distance $0$ or $2$ from each other, respectively. Set
\begin{equation}
\label{eq: p_0,p_2}
\begin{split}
p_0&:= \sum_{j \ge 1} (d/a')^{j}d^{-j+1}=d \frac{1}{a'-1}, \\ p_2&:=
\sum_{j \ge 1} \left(\frac{3d}{a'(2d-7)}
\right)^{j} (3/d)^{j-1} =\left(\frac{3d}{a'(2d-7)}
\right)\frac{1}{1- \frac{9}{a'(2d-7)}}.
\end{split}
\end{equation}
 We now analyze the contribution to $\rho_k$ coming from the time intervals in which the paths are at distance $0$ or $2$ from each other, respectively. Again, we employ the random walk interpretation and think about $\gamma$ and $\gamma'$ as being exposed one vertex at a time. By (\ref{eq: pathsintersectionprobability}) and the aforementioned stochastic dominations by Geometric random variables, each time interval in which the paths
are at distance $0$ or $2$ contributes to $\rho_{k}$ a multiplicative term of at
most $p_0$ or $p_2$, respectively.
Moreover, since in the case $k=\infty$, given $(s_0,s_2)$, the lengths
of the different time intervals are independent, then given
$(s_0,s_2)$, the total contribution
of all the different time intervals is bounded by $p_0^{s_0}p_2^{s_2}$.  Whence, in order to bound $w_k$ it suffices to bound $\mathbf{E}_{2k-1} \times \mathbf{E}_{2k-1}  \left[p_0^{s_0(\gamma,\gamma')}p_2^{s_2(\gamma,\gamma')}\right]$.
 
Let $q_{2,0}, q_{2,4},q_{4,2}$ be as above. Consider the following Markov chain with state space $\{0,2,4,\infty \}$,  whose initial
state is $0$. If the current state of the chain is $0$, then the next state is $2$
$\mathrm{w.p.}$ 1. If the current state of the chain is $2$, then the next state is $0$ $\mathrm{w.p.}$ $q_{2,0}=d^{-2}+O(d^{-3})$ and
otherwise ($\mathrm{w.p.}$ $q_{2,4}$) it is $4$. If the current state of the chain is $4$, then $\mathrm{w.p.}$
$q_{4,\infty}:=1-q_{4,2}=1-4d^{-2}+O(d^{-3})$ the chain moves to an absorbing state $\infty$ and otherwise ($\mathrm{w.p.}$ $q_{4,2}$)
it moves to state $2$. Let $r_0,r_2$ be the number of times this chain visits states $0$ and $2$, respectively, before getting absorbed in $\infty$.  Note that for  any $1 \le k_0 \le k_2 \in \N$
\begin{equation}
\label{estimatingbychain}
\begin{split}
& \Pr[r_0=k_0,r_2=k_2]=q_{2,0}^{k_0-1}q_{4,2}^{k_2-k_0}q_{4,\infty} \\ & =[d^{-2}+O(d^{-3})]^{k_0-1} [4d^{-2}+O(d^{-3})]^{k_2-k_0}[1-4d^{-2}+O(d^{-3})].
\end{split}
\end{equation}
By
the above analysis, $(r_0,r_2)$ stochastically
dominates $(s_0(\gamma,\gamma'),s_2(\gamma,\gamma'))$, when $(\gamma,\gamma')$
is picked according to $\nu_{2k-1} \times \nu_{2k-1}$. Let  $d$ be sufficiently large so that $p_0,p_2 > 1$. Using the aforementioned stochastic
domination on the increasing function $f(x_{0},x_{2})=p_0^{x_0}p_2^{x_2}$  in conjunction
with (\ref{eq: p_0,p_2}) and (\ref{estimatingbychain}), we get that
\begin{equation}
\label{eq: w(gamma,gamm')}
\rho_{k} \le \sum_{k_0,k_2}p_0^{k_0}p_2^{k_{2}}\Pr[s_0=k_0,s_2=k_2] \le \sum_{k_0,k_2}p_0^{k_0}p_2^{k_2}\Pr[r_0=k_0,r_2=k_2] \le bp_0p_2=:\beta,
\end{equation}
for some absolute constant $b$. This was noted earlier to imply (\ref{eq: CS}), which concludes the proof by the paragraph following (\ref{eq: CS}).  \qed
\section{The first 3 layers in random graphs}
In this section we prove Theorem \ref{thm : randomgraphs}. Our approach utilizes an auxiliary random graph, similar to an auxiliary construction considered in \cite{feige2013layers}.
\label{sec: expanders}
\subsection{Preliminaries}
A connected component that contains a linear fraction of the vertices of
a graph is commonly referred to as a {\em giant component}.  There has been extensive work on the formation of giant components
in random graph models and in percolation on random graphs (see for example~\cite{Fount,KSud,Janson,Molloy}).

We now state a few
basic facts about \emph{random graphs with a given degree sequence}. Let $\bar{d}$ be a non-increasing sequence of $n$
nonnegative integers and let $d_i$ be the $i$th element of $\bar{d}$. Throughout, we
assume that $\sum d_i$ is even and that $d_1 \le d$ for some $d$ which does not
depend on $n$.

We consider the random graph model $\mathcal{G}_{n,\bar d}$ in which a graph is sampled according to the uniform distribution over all graphs labeled by the set $[n]$ in which vertex $i$ has degree $d_i$. We indicate that a random graph $G$ has such a distribution by
writing $G \sim \mathcal{G}_{n,\bar
d}$. If  $G \sim \mathcal{G}_{n,\bar
d}$ we say that $G$ is a \emph{random graph with a given degree sequence}
$\bar d$.  In practice, we consider a sequence of degree sequences $\bar d(n)$, where $\bar d(n)=(d_1(n),\ldots,d_{k_{n}}(n))$, and a sequence of graphs $G_n \sim \mathcal{G}_{k_{n},\bar d(n)}$, where $\lim_{n \to \infty}k_n=\infty$. We hide the dependence on $n$ from our notation.  

An intimately related model to $\mathcal{G}_{n,\bar
d}$ is the \emph{configuration model} $\mathcal{P}_{n,\bar{d}}$, which is a model for generating a random multigraph whose vertex set
is labeled by $[n]$ and vertex $i$ has $d_i$ ``half-edges". The half-edges are
combined into edges by choosing uniformly at random a matching of all ``half-edges". We indicate that a random multigraph $G$ has such
a distribution by writing $G \sim \mathcal{P}_{n,\bar
d}$.

Given a multigraph sampled according to the configuration model, Molloy and
Reed~\cite{Molloy} provide a criterion for the existence of a giant component. We refer to the condition in the Theorem \ref{thm:MolloyReed} as the \emph{Molloy-Reed condition}.
The exact statement of their result involves some technical conditions and
parameters that are omitted here due to our bounded degree setup.
\begin{theorem}
\label{thm:MolloyReed}
Let $(\bar d(n):n \in \N)$ be a sequence of degree sequences. Assume that $\sup_{n}d_1(n)<\infty$.  Denote by $\lambda_i(n)$
 the fraction of vertices of degree $i$ in $\bar d(n)$. Let \[Q(\bar{d}(n)) = \sum_{i
\ge 1} \lambda_i(n) i(i - 2).\] Assume further that $Q(\bar d(n)) > \epsilon >0$ for all sufficiently large $n$. Let $G_n \sim \mathcal{P}_{n,\bar d(n)}$. Then the probability
of  $G_{n}$ not having a giant component is exponentially small in $n$.
\end{theorem}
We now state a few elementary facts about $\G_{n,\bar d}$ and $\mathcal{P}_{n,\bar d}$ and their relations (see e.g$.$ \cite{wormald} Theorem 2.6). We say that a set of variables $X_i^{(n)}$ for $i$ in some finite set $I$, defined on
a sequence of probability spaces indexed by $n$ are \emph{asymptotically independent Poisson} with means $\lambda_i$ if their joint distribution tends to that of independent
Poisson variables whose means are fixed numbers $\lambda_i$. When $\bar d=(d,d,\ldots,d)$ is the fixed sequence, we write  $\mathcal{G}_{n,d}$ and  $\mathcal{P}_{n,d}$ instead of   $\mathcal{G}_{n,\bar
d}$ and $\mathcal{P}_{n,\bar d}$, respectively.

\begin{theorem}
\label{thm: cycledist}
\begin{itemize}
\item[(i)]
Let $d \ge 3$. Let $X_i^{(n)}$ be the number of cycles of length $i$ in  $G_{n,d} \sim \mathcal{P}_{n,d}$. Then $X_1^{(n)},\ldots, X_k^{(n)}$ are asymptotically independent Poisson with means $\lambda_i:=(d-1)^{i}/2i$, for all  $k \ge 1$. \item[(ii)]
Let $\bar d=(d_1,\ldots,d_n)$ be a degree sequence with $d_1 \le d$. Let $Y_i^{(n)}$ be the number of cycles
of length $i$ in $G_{n,\bar d} \sim \mathcal{P}_{n,\bar d}$. Then for all $n$ and  $k \ge 1$, the joint distribution of $Y_1^{(n)},\ldots,Y_k^{(n)}$ is stochastically dominated by that of  $X_1^{(n)},\ldots, X_k^{(n)}$ from part (i).
\item[(iii)] Let   $G_{n,\bar d} \sim \mathcal{P}_{n,\bar d}$. Let $Y_i^{(n)}$ be as in (ii). Given $Y_1=Y_2=0$, we have that $G_{n,\bar d} \sim \mathcal{G}_{n,\bar d}$. 
\end{itemize}
\end{theorem}
\subsection{Proof of Theorem \ref{thm : randomgraphs}}
\emph{Proof:}
Fix some odd $k \ge 15$ to be determined later. Let $G=([n],E) \sim \mathcal{G}_{n,\bar d} $ for some degree sequence $\bar d:=(d_1,\ldots,d_n) $ with $d_1 \le d$.  We first argue that  with probability $1-\exp(-\Omega(n))$ there are less than $\frac{n}{4d^{k+2}}$ vertices belonging to a cycle of length at most $k+2$ in $G$. Call this event $F$. By part (iii) of Theorem \ref{thm: cycledist} it suffices to show that the same holds for $H' \sim \mathcal{P}_{n,\bar d}  $.
By Theorem \ref{thm: cycledist} the expected number of vertices belonging to a cycle of length at most $k+2$ is $O(1)$. Consider the exposure process of $H'$ in which the half edges are exposed one at a time. Consider the Doob's martingale of the number of vertices belonging to a cycle of length at most $k+2$ w.r.t.~this exposure procedure. The increments of this martingale can be bounded by $d^{k+2}$. Indeed, using the description of the configuration model, the increments can be bounded by the maximal change in the number of vertices belonging to a cycle of length at most $k+2$ possible by the operation of  adding a single edge to a graph with degree sequence $\bar d$. Indeed, to control the increments,  after exposing a certain edge by matching together its two half-edges, we can now increase the number of half-edges at the corresponding vertices by one (or 2 if the exposed edge was a self loop) and the distribution of the remaining edges will remain unaffected. This idea can be used for bounding the increments both from above and below (we leave the details as an exercise).  Hence, using Azuma inequality we obtain the desired bound. 

On the event $F$, there exist a constant $c=c(d)$ and a set $I \subset [n]$ with $|I| \ge cn$ such that every two vertices in $I$ are of distance at least $15$ from one another and the ball of radius $k+2$ around each of them does not contain a cycle.  Thus on the event $F$, there are two sets $A,B$  both of size  $\lceil cn \rceil$ such that every vertex in $A$ is adjacent to exactly one vertex in $B$ (and vice versa) and the ball of radius $k+1$ around every vertex in $A \cup B$ does not contain a cycle. 

Assuming the event $F$ occurs, fix such sets
$A,B$. Consider the  graph $G'$ obtained by removing from $G$ all the edges between $A$ and $B$. Denote the set of deleted edges by $M$. Note that given $G'$ and the set $D:=A \cup B$, the conditional distribution of $M$ is uniformly distributed among all perfect matchings of the set $D$ (we assume $G,A$ and $B$ are unknown and only $G'$ and $D=A \cup B$ were exposed).

Consider the  vertex-induced random subgraph of $G'$ (induced w.r.t$.$ $G'$) on the vertex set $$\{u \in [n] \setminus D: u \in T_3(G')\} \cup \{u \in D: u \in T_2(G') \}, \text{ denoted by $R$}.$$ We use the same source of randomness to generate the ages of the vertices in $T_3(G)$ and in $T_3(G')$. Note that for $u \in [n] \setminus D$, we have that $ u \in T_3(G')$ iff $ u \in T_3(G)$ as  the set of neighbors of $u$ in $G$ is the same as the set of its neighbors in $G'$. Also, note that since for every $u \in D$ the set of neighbors of $u$ in $G'$ is contained in the set of its neighbors in $G$ and is smaller by one neighbor. Hence if $u \in T_2(G')$, then it must be the case that $u \in T_3(G)$.

Consider the following (random) set of edges $M':=\{\{v,u\} \in M: v,u \in R \}$. The requirement $\{v,u\} \in M$ implies that $v,u \in D$ and thus the requirement that $v,u \in R$ implies that they belong to $T_2(G')$.
Whence $R \cup M' \subset T_3(G)$, where by $R \cup M'$ we mean the graph obtained by adding the edges of $M'$ to $R$. Hence it suffices to show that $R \cup M'$ has a giant component with probability
$1-\exp(-\Omega(n))$. 

For every $u \in D$ let $C_u'$ be the intersection of the connected component of $u$ in $R$ with the set $\{v:d_{G'}(u,v) \le k\}$.  Notice that for every $u \in D$, we have that $C_u'$
 depends only on the ages of
the vertices of distance at most $k+1$ from $u$ (distance w.r.t$.$ the graph $G'$). Since (by construction) up to a distance
of $k+1$ from any $u \in D$ the graph looks  like a tree, by Lemma \ref{lem: auxforsec6} there exist some constants $\bar c>0$ and $a>1$ (both independent of $k$) such that every $u \in D$ satisfies $|C_u'| \ge a^k$ with probability at least $\bar cd^{-1}$. We call a vertex $u \in D$ which satisfies  $|C_u'| \ge a^k$ \emph{good}. The occurrence or non-occurrence of the event that a vertex $v \in D$ is \emph{good} can affect the conditional probability of the event that $u \in D$ is \emph{good} only for those $u \in D$ such that $d_{G'}(v,u) \le 3k$. Denote the set of good vertices by $Good$ ($ \subset D$). Let $D( \ell)$ be the event that $|Good| \ge \ell$. By the assumption that the bound on the maximal degree, $d$, is a constant, Azuma inequality (Theorem \ref{thm: azuma}) implies that for some constant $c'=c'(d)>0$ (independent of $k$)
\begin{equation}
\label{eq: D(k)}
\Pr[D(c'n) \mid F ] \ge 1-\exp[-\Omega(n)].
\end{equation}
On the event $D(c'n) \cap F$ we have a collection of at most $n/a^{k}$  connected components of $R$ each of size at least $a^{k}$ and the union of these connected components contains at least $c'n$ vertices of $D$. If a connected component is of size greater than $a^{k}$, then in what comes we artificially treat it as several distinct connected components each of size between $a^k$ and $2a^k$ by partitioning it in a arbitrary manner into (disjoint) sets of sizes between $a^k$ and $2a^k$. Call the collection of all ``large" components of $R$ we obtain in this manner $L$ (where we say a component is ``large" if its size is between $a^k$ and $2a^k$).
Note that
\begin{equation}
\label{eq: sizeofL}
c'na^{-k} \le |L| \le na^{-k}.
\end{equation}

Construct now an auxiliary (random) multigraph $H$ with two sets of vertices,
$U_1$ and $U_2$. Every component in $L$ serves as a vertex
in $U_1$, and the number of \emph{good} $D$ vertices in a component serves as the number
of half-edges of the corresponding $U_1$-vertex (to avoid confusion we shall
refer to the vertices in $U_1$ as super-vertices (following the terminology of \cite{feige2013layers}). Hence, on $D(c'n) \cap F$, we have that $c'n/a^k \le |U_1| \le n/a^{k}$
and as every good vertex in $D$ within a component in $L$ contributes
one edge to the degree of the super-vertex of $H$ corresponding to its component in $T_3(G')$, we get that
\begin{equation}
\label{eq: avedegU1}
\sum_{u \in U_1} d_u \ge c'n.
\end{equation}
 The set $U_2$ consists of all
the non-good vertices in $D$. Every vertex in $U_2$ has degree $1$.

Consider now the configuration model for generating random multigraphs with vertex
set and number of half-edges for each vertex as described above and call
the resulting random multigraph $H$. The distribution of $G$ given $G'$ and $D$ (conditioned on the event $D(c'n) \cap F$ whose probability is $1-\exp(-\Omega(n))$) can be coupled with $H$ in a natural manner. Since the number of vertices of $H$ (on the event $D(c'n) \cap F$) is $\Omega (n)$, if $H$ contains a giant component
then so does $T_3(G)$. Namely, a giant component
in $H$ of size at least $ d\alpha n$ implies a giant component in $T_3(G)$ of size at least $\alpha n$.

To determine that $H$ is likely
to have a giant component we use Theorem \ref{thm:MolloyReed}. We first need the following proposition taken from \cite{feige2013layers}.
\begin{proposition}
\label{pro:equal}
Consider two vertices of degree $r$ and $r' \ge r + 2$. Then the expression
$\sum_{i > 1} \alpha_i d_i(d_i - 2) > 0$ decreases by replacing them by vertices
of degrees $r + 1$ and $r' - 1$.
\end{proposition}
\begin{proof}
Initially the contribution of the two vertices is $r(r-2) + r'(r'-2)$. After
replacement it is $(r+1)(r-1) + (r'-1)(r'-3)$, which is smaller by $2(r'
- r - 1)$.
\end{proof}
Observe that on $D(c'n) \cap F$, by (\ref{eq: sizeofL}) and (\ref{eq: avedegU1}) the average degree of a super-vertex in $U_1$ is at least $c_1 a^{k}$ and the fraction of $U_1$ vertices in $H$ is at least $c_2 a^{-k}$ for some constants $c_1,c_{2}>0$ independent of $k$. Let $\ell(k):= \lfloor c_1 a^k \rfloor$. We denote the degree sequence of $H=(V',E')$ by $d_1(H),\ldots,d_{|V'|}(H)$. The above proposition, together with the fact that the expression $\sum_{i \ge 0}\alpha_i d_i(H) (d_i(H)-2)$ is a monotone function of the $d_i(H)$'s, implies that for the graph $H$, the expression $\sum_{i
\ge 0}\alpha_i d_i(H) (d_i(H)-2)$ is minimized (on $D(c'n) \cap F$) when all vertices of $U_1$ are of degree $\ell(k)$. The fraction of $U_2$ vertices of $H$ can be  bounded from above by 1.   The Molloy-Reed condition is indeed satisfied whenever $k$ is sufficiently large, as
$$  \sum_{i
\ge 0}\alpha_i d_i(H) (d_i(H)-2) \ge 1 \cdot (-1)+ c_2 a^{-k} \ell(k) (\ell(k)-2)>\epsilon, $$  for some constant $\epsilon>0$. So if we fix such $k$ we have that the Molloy-Reed condition indeed holds for $H$.
\qed
\subsection{Possible extensions}
We now discuss some
technicalities related to possible relaxations of the assumptions on the degree sequence in Theorem \ref{thm : randomgraphs}. We then discuss the connection of Theorem \ref{thm : randomgraphs} (and the aforementioned relaxations) to the analysis of $T_3(G(n,c/n))$.
We end by discussing an alternative approach that can cover the more general case of large girth expanders and some of the difficulties related to that approach.

The assumption in Theorem \ref{thm : randomgraphs} that the minimal degree
is at least 3 could be relaxed to allow degree 2 vertices and is present
mostly for convenience. To see this, observe that the existence of degree
2 vertices (as long as we also have some fraction of vertices of strictly
larger degree) in our random graphs setup, would still allow one to get a
parallel version of Lemma \ref{lem: auxforsec6} with smaller $a>1$ and larger
$k$. The assumption that there are no degree 1 vertices is more substantial,
though with some work the proof can be adapted to the situation where the
fraction of such vertices is sufficiently small, although this would complicate
manners substantially. The difficulty would then be showing that a parallel
statement to Lemma \ref{lem: auxforsec6} still holds in such a setup. 

Hence
it is reasonable to expect that if $c>1$ is sufficiently large
and $G \sim G(n,\frac{c}{n})$ (i.e$.$ $G$ is an Erd\H os-R\'enyi random graph
with parameter $c/n$), then $T_3(G)$ has a giant component w.h.p$.$. Note
that $G$ would not be of bounded degree, but this turns out to not be a major difficulty (for sufficiently large $d=d(c)$
 and $k=k(c) \in \N$, the number of vertices which would not have a vertex
of degree more than $d$ in a ball of radius $k$ around them in $G$ is w.h.p$.$
linear in $n$). We know that when $c-1$ is positive but sufficiently small, then w.h.p$.$
all components of $T_3(G)$ are of size $O(\log n)$. Whence one should expect
a phase transition with respect to $c$ for the existence of a giant component
in $T_3(G)$. It is left as an open problem to verify this and to find the
critical $c$. We note that a-priori it is not clear that there is monotonicity
w.r.t$.$ $c$ for the existence of a giant component w.h.p$.$ in $T_3(G)$,
but we believe this is indeed the case.
\begin{question}
Let $G \sim G(n,c/n)$. What is the critical $c$ for the existence of a giant component in $T_3(G)$?
\end{question}

\section*{Acknowledgements}
The author would like to thank Itai Benjamini, Uriel Feige, Daniel Reichman  and Allan Sly for many useful discussions.

\bibliographystyle{plain}
\bibliography{Layersproject}
\section{Appendix}
\label{sec: A}
In this section we state a few theorems that where used in this work. We start with Azuma inequality (see e.g$.$ Theorem 13.2 in \cite{lyons2005probability}).
\begin{lemma}[\textbf{Azuma inequality}] \label{thm: azuma}
Let $X_0,...,X_n$ be a martingale such that for every $1 \leq k <n$ it holds
that $|X_k-X_{k-1}|\leq c_k$. Then for every nonnegative integer $t$ and
real $B >0$
$$\Pr(X_t-X_0 \leq -B) \leq \exp\left(\frac{-B^2}{\sum_{i=1}^tc_i^2}\right).$$
\end{lemma}

The following Theorem is due to Paley and Zygmund (see e.g$.$ \cite{lyons2005probability} pg 162)
\begin{theorem}[\textbf{Paley and Zygmund's inequality}]
\label{thm: payleyzygmund}
If $X$ is a random variable with mean 1 and $0<t <1$, then $$\Pr(X >t) \ge \frac{(1-t)^{2}}{
 \mathbf{E}[X^2]}.$$
\end{theorem}

Suppose that we have a countable (possibly finite) set of vertices, $S$,
and
a state space $\Omega:=\R^{S}$, with the product topology and the corresponding
Borel $\sigma$-algebra $\mathcal{B}$. We now define the notion of \emph{stochastic
domination}. Given $\omega_1,\omega_2 \in \Omega$, we say that $\omega_1
\leqslant \omega_2$, whenever $\omega_1(s)
\leqslant \omega_2(s)$ for any $s \in S$.
We say that a measurable function $f:\Omega \to \R$ is \emph{increasing},
whenever $\omega_1
\leqslant \omega_2 \in \Omega$ implies $f(\omega_1) \le f(\omega_2)$. When
we have real valued random variables $X=(X_{s} : s \in S)$, we sometimes
say an event is increasing with respect to $X$. By this we mean that the
indicator function of that event may be written as $f(X)$ for some increasing
function $f:\Omega \to \R$. We say it is \emph{decreasing} when $-f$ is increasing.
Given
two Borel probability measures on $\Omega$, $\mu$ and $\nu$, we say that
$\mu$ \emph{stochastically dominates} $\nu$, and denote this by $\mu \geqslant
\nu$, if for any continuous increasing function $f$ we have
 $$\int_{\Omega}fd\mu
\ge \int_{\Omega}fd\nu.$$
We say that an event is \emph{increasing} (respectively, \emph{decreasing})
if its indicator function is  \emph{increasing} (respectively, \emph{decreasing}).
It is not hard to prove that it suffices to restrict to the case that $f$
is an indicator of an increasing event. Moreover, by a simple limiting argument,
it suffices to consider increasing indicators of events that depends only
on finitely many co-ordinates. 

The following Theorem is due to Harris (1960). It is often referred to as
FKG inequality,
due to a generalization due to  Fortuin, Kasteleyn and Ginibre (1971). For
a proof, see e.g. Section 2.2. of $\cite{grimmett1999percolation}$.
\begin{theorem}[Harris/FKG inequality]
\label{thm: FKGinequality}
Suppose that we have a countable (possibly finite) set of vertices, $S$,
and a state space $\Omega=\R^{S}$. Assume we have independent $\R$ valued
random variables $X:=(X_{s} :s \in S)$. Let $f,g:\Omega \to \R$ be increasing,
then
$$\mathbf{E}[f(X)g(X)] \ge \mathbf{E}[f(X)]\mathbf{E}[g(X)].$$
The same holds when both $f$ and $g$ are decreasing. If $f$ is increasing
and $g$ is decreasing the inequality holds in the reverse direction.
\end{theorem}
We now present a more general Theorem we shall need. Let $k \in \N$. For
any $x,y \in \R^k$ denote by $x \wedge y$ and $x \vee y$ their coordinate-wise
minimum
and maximum, respectively. A set $S \subset \R^k$ is called a sub-lattice
if whenever $x,y \in S$ so are  $x \wedge y$ and $x \vee y$.
\begin{theorem}
\label{thm: affiliation}
Let $X:=(X_{1},\ldots,X_k)$ be independent real valued random variables.
Let $S \subset \R^k$ be a sub-lattice so that $\Pr[X \in S]>0$. Let $A,B$
be two increasing events with respect to $X$. Then,
\begin{equation}
\label{eq: affiliationinequality}
\Pr[A \cap B |X \in S] \ge \Pr[A |X \in S]\Pr[B |X \in S].
\end{equation}   
\end{theorem}
Random variables that satisfy (\ref{eq: affiliationinequality}) are called
\emph{affiliated}. It is well-known (e.g.\ the appendix of $\cite{milgrom1982theory}$)
that if $(X_{1},\ldots,X_k)$ have a joint density function $f$ which satisfies
\begin{equation}
\label{eq: affiliateddensityfunc}
f(x \vee y)f(x \wedge y) \ge f(x)f(y),
\end{equation}
 for  Lebesgue $\mathrm{a.e.}$ $(x,y) \in \R^{2k}$, then $(X_{1},\ldots,X_k)$ are \emph{affiliated}.
We call such a non-negative function satisfying (\ref{eq: affiliateddensityfunc})
 $\emph{affiliated}$. It is easy to verify that if $f(z_1,z_2)=g(z_1)h(z_2)$
and $g$ and $h$ are \emph{affiliated}, then so is $f$. Since any non-negative
$g : \R \to \R$ is trivially affiliated and a joint density function of independent
random variables factors to a product of the marginal densities, we indeed
get that independent random variables are always \emph{affiliated} as stated
in Theorem \ref{thm: affiliation}. 

\end{document}